\documentclass{article}
\usepackage[utf8]{inputenc}
\usepackage[english]{babel}
 
\usepackage[square,sort,comma,numbers]{natbib}
\usepackage{pigpen}
\newcommand{\po}{\ar@{}[dr]|{\text{\pigpenfont R}}}
\newcommand{\pb}{\ar@{}[dr]|{\text{\pigpenfont J}}}
\usepackage{graphicx}
\usepackage{tikz}
\usetikzlibrary{arrows}
\usetikzlibrary{matrix,arrows,decorations.pathmorphing}

\usepackage{graphicx}
\usepackage{tikz}
\usetikzlibrary{arrows}
\usetikzlibrary{matrix,arrows,decorations.pathmorphing}

\usepackage{amsmath}
\usepackage{amssymb}
\usepackage{amsthm}
\usepackage{float}

\numberwithin{equation}{section}
\numberwithin{figure}{section}

\usepackage{stmaryrd}
\usepackage[all]{xy}
\usepackage[mathcal]{eucal}
\usepackage{verbatim} 
\usepackage{fullpage} 
\usepackage{wrapfig}
\usepackage{lipsum}
\usepackage[all]{xy}
\theoremstyle{plain}
\newtheorem{thm}{Theorem}[section]

\newtheorem{lem}[thm]{Lemma}
\newtheorem{prop}[thm]{Proposition}
\newtheorem{cor}[thm]{Corollary}
\newtheorem{claim}[thm]{Claim}

\newtheorem{que}[thm]{Problem}

\theoremstyle{definition}
\newtheorem{defn}[thm]{Definition}
\newtheorem{remark}[thm]{Remark}

\title{The universal surface bundle over the Torelli space has no sections}
\author{Lei Chen}

\begin{document}
 \bibliographystyle{alpha}
\maketitle

\begin{abstract}
For $g>3$, we give two proofs of the fact that the \emph{Birman exact sequence} for the Torelli group
\[
1\to \pi_1(S_g)\to {\cal I}_{g,1}\to {\cal I}_g\to 1
\] does not split. This result was claimed by G. Mess in \cite{mess1990unit}, but his proof has a critical and unrepairable error which will be discussed in the introduction. Let ${\cal UI}_{g,n}\xrightarrow{Tu'_{g,n}} {\cal BI}_{g,n}$ (resp. ${\cal UPI}_{g,n}\xrightarrow{Tu_{g,n}}{\cal BPI}_{g,n}$)  denote the universal surface bundle over the Torelli space fixing $n$ points as a set (resp. pointwise). We also deduce that $Tu'_{g,n}$ has no sections when $n>1$ and that $Tu_{g,n}$ has precisely $n$ distinct sections for $n\ge 0$ up to homotopy.
\end{abstract}

\section{Introduction}

It is a basic problem to understand when bundles have continuous sections, and the corresponding group theory problem as to when short exact sequences have splittings. These are equivalent problems when the fiber, the base and the total space are all $K(\pi,1)$-spaces. In this article, we will discuss the ``section problems" and the ``splitting problems" in the setting of surface bundles. Here by \emph{section} we mean continuous section.

Let $S_{g,n}$ be a closed orientable surface of genus $g$ with $n$ punctures. Let Mod$_{g,n}$ (resp. PMod$_{g,n}$) be the \emph{mapping class group} of $S_{g,n}$, i.e. the group of isotopy classes of orientation-preserving diffeomorphisms of $S_g$ fixing n points as a set (resp. pointwise). Mod$_{g,n}$ and PMod$_{g,n}$ act on $H^1(S_g;\mathbb{Z})$ leaving invariant the algebraic intersection numbers. Let ${\cal I}_{g,n}$ (resp. ${\cal PI}_{g,n}$) be the \emph{Torelli group} (resp. \emph{pure Torelli group}) of $S_{g,n}$, i.e. the subgroup of Mod$_{g,n}$ (resp. PMod$_{g,n}$) that acts trivially on $H^1(S_g;\mathbb{Z})$. We omit $n$ when $n=0$. The following \emph{Birman exact sequence} for the Torelli group provides a relationship between ${\cal I}_{g,1}$ and ${\cal I}_{g}$; see \cite[Chapter 4.2]{BensonMargalit}.
\begin{equation}
1\to \pi_1(S_g)\xrightarrow{\text{point pushing}} {\cal I}_{g,1}\xrightarrow{T\pi_{g,1}} {\cal I}_{g}\to 1.
\label{*}
\end{equation}

The main theorem of this paper is the following:
\begin{thm}[\bf Nonsplitting of the Birman exact sequence for the Torelli group]
For $g>3$, the Birman exact sequence for the Torelli group (\ref{*}) does not split.
\label{main}
\end{thm}
\begin{remark}Our proof needs the condition $g>3$. By \cite[Proposition 4]{MR1191379}, ${\cal I}_2$ is a free group. So the Birman exact sequence for ${\cal I}_2$ splits. The case $g=3$ is open.\end{remark}

Let ${\cal BPI}_{g,n}:=K({\cal PI}_{g,n},1)$ be the \emph{pure universal Torelli space} fixing $n$ punctures pointwise and let 
\begin{equation}
S_g\to {\cal UPI}_{g,n}\xrightarrow{Tu_{g,n}}{\cal BPI}_{g,n}
\label{TUB2}
\end{equation}
 be the \emph{pure universal Torelli bundle}. Surface bundle (\ref{TUB2}) classifies smooth $S_g$-bundle equipped with a basis of $H^1(S_g;\mathbb{Z})$ and $n$ ordered points on each fiber. Since ${\cal PI}_{g,n}$ fixes $n$ points, there are $n$ distinct sections $\{Ts_i|1\le i\le n\}$ of the universal Torelli bundle (\ref{TUB2}). Let ${\cal BI}_{g,n}:=K({\cal I}_{g,n},1)$ be the \emph{universal Torelli space} fixing $n$ punctures as a set and let 
\begin{equation}
S_g\to {\cal UI}_{g,n}\xrightarrow{Tu'_{g,n}} {\cal BI}_{g,n}
\label{TUB1}
\end{equation}
 be the \emph{universal Torelli bundle}. This bundle classifies smooth $S_g$-bundles equipped with a basis of $H^1(S_g;\mathbb{Z})$ and $n$ unordered points on each fiber. Theorem \ref{main} says that $Tu_{g,0}$ has no sections. For $n\ge 0$, we have the following complete answer for sections of \eqref{TUB1} and \eqref{TUB2}.
\begin{thm}[\bf Classification of sections for punctured Torelli spaces]
The following holds:\\
(1) For $n\ge 0$ and $g>3$, every section of the universal Torelli bundle (\ref{TUB2}) is homotopic to $Ts_i$ for some $i\in \{1,2,...,n\}$.  \\
(2) For $n>1$ and $g>3$, the universal Torelli bundle (\ref{TUB1}) has no continuous sections.
\label{Torelli}
\end{thm}

Let ${\cal M}_g:=K(\text{Mod}_g,1)$. As is known, the universal bundle over ${\cal M}_g$:
\[S_g\to {\cal UM}_g\xrightarrow{} {\cal M}_g\] has no sections. This can be seen from the corresponding algebraic problem of finding splittings of the \emph{Birman exact sequence}
\[
1\to \pi_1(S_g)\to \text{Mod}_{g,1}\to \text{Mod}_{g}\to 1.
\]
The answer is no because of torsion, e.g. see \cite[Corollary 5.11]{BensonMargalit}. The key fact is that every finite subgroup of Mod$_{g,1}$ is cyclic. However, this method does not work for torsion-free subgroups of Mod$_g$. For any subgroup $G<$ Mod$_g$, there is an extension $\Gamma_G$ of $G$ by $\pi_1(S_g)$ as the following short exact sequence. 
\begin{equation}
1\to \pi_1(S_g)\to \Gamma_G\to G\to 1.
\label{BESFG}
\end{equation}
We call \eqref{BESFG} the \emph{Birman exact sequence} for $G$ since it is induced from the Birman exact sequence. We pose the following open question: 
\begin{que}[\bf Virtually splitting of the Birman exact sequence]
Does the Birman exact sequence for a finite index subgroup of \normalfont{Mod}$(S_g)$ always not split?
\end{que}
Let the \emph{level L congruence subgroup} Mod$_g[L]$ be the subgroup of Mod$_g$ that acts trivially on $H_1(S_g;\mathbb{Z}/L\mathbb{Z})$ for some integer $L>1$. Theorem \ref{main} implies that a finite index subgroup of Mod$_g$ containing ${\cal I}_g$ does not split; in particular, this applies to all the congruence subgroups Mod$_g[L]$.
\\
\\
\medskip
{\large\bf Error in G. Mess \cite[Proposition 2]{mess1990unit}}

In the unpublished paper of G. Mess \cite[Proposition 2]{mess1990unit}, he claimed that there are no splittings of the exact sequence (\ref{*}). But his proof has a fatal error. Here is how the proof goes. Let $C$ be a curve dividing $S_g$ into $2$ parts $S(1)$ and $S(2)$ of genus $p$ and $q$, where $p,q\ge 2$. Let $UTS_g$ be  the unit tangent bundle of genus $g$ surface. Then ${\cal I}_g$ contains a subgroup $A$, which satisfies the following exact sequence
\[
1\to \mathbb{Z}\to \pi_1(UTS_p)\times \pi_1(UTS_q)\to A\to 1.\]
Mess' idea is to prove that the Birman exact sequence for  $A$ does not lift. However, in Case a) of Mess' proof for \cite[Proposition 2]{mess1990unit}, Mess claimed that if the Dehn twist about $C$ lifts to a Dehn twist about $C'$ on $S_{g,1}$ and $C'$ bounds a genus $p$ surface with a puncture and a genus $q$ surface, then there is a lift from $\pi_1(UTS_p)$ to $\pi_1(UTS_{p,1})$. This is a wrong claim. Actually even $A$ does have a lift. We construct a lift of $A$ as the following. Let $\text{PConf}_2(S_p)$ be the \emph{pure configuration space} of $S_p$, i.e. the space of $2$-tuples of distinct points on $S_p$. Let 
\[
\text{PConf}_{1,1}(S_p)=\{(x,y,v)|x\neq y\in S_g \text{ and } v \text{ a unit vector at $x$}\}.
\]
We have the following pullback diagram:
\begin{equation}
\xymatrix{
 \pi_1(\text{PConf}_{1,1}(S_p))\ar[r]\ar[d]^f   \pb  & \pi_1(\text{PConf}_2(S_p))\ar[d]^g\\
   \pi_1(\text{UT}S_p)\ar[r]    & \pi_1(S_p)   .}
\label{Mess}
\end{equation}
The lift of $\pi_1(UTS_p)$ should lie in $\pi_1(\text{PConf}_{1,1}(S_p))$ instead of $\pi_1(UTS_{p,1})$ as Mess claimed. As long as we can find a lift of $f$, we will find a section of $A$ to ${\cal I}_{g,*}$. By the property of pullback diagrams, a section of $g$ can induce a section of $f$ in diagram \eqref{Mess}. To negate the argument of \cite[Proposition 2]{mess1990unit}, we only need to construct a section of $g$. We simply need to find a self-map of $S_g$ that has no fixed point. For example, the composition of a retraction of $S_p$ onto a curve $c$ and a rotation of $c$ at any nontrivial angle does not have a fixed point. Therefore Mess' proof is invalid and does not seem to be repairable.
\\
\\
\medskip
{\large\bf Strategy of the proof of Theorem \ref{main}}

 Let $T_a$ be the Dehn twist about a simple closed curve $a$ on $S_g$. Our strategy is the following: assume that we have a splitting of (\ref{*}). The main result of \cite{johnson1983structure} shows that all \emph{bounding pair map} i.e. $T_aT_b^{-1}$ for a pair of nonseparating curves $a,b$ that bound a subspace generate ${\cal I}_{g}$. Firstly we need to understand the lift of $T_aT_b^{-1}$. We show that the lift of a bounding pair $T_aT_b^{-1}$ has to be a bounding pair $T_{a'}T_{b'}^{-1}$ for $a',b'$ on $S_{g,1}$. Moreover, the curve $a'$ does not depend on the choice of $b$, i.e. for any other curve $c$ that forms a bounding pair with curve $a$, the lift of $T_aT_c^{-1}$ is $T_{a'}T_{c'}^{-1}$ for the same $a'$. Therefore, we have a lift from the set of isotopy classes of curves on $S_g$ to the set of isotopy classes of curves on $S_{g,1}$. Then we use the lantern relation to derive a contradiction. Our main tool is the canonical reduction system for a mapping class, which in turn uses the Thurston classification of isotopy classes of diffeomorphisms of surfaces. This idea originated from \cite{MR726319}.
\\
\\
\medskip
{\large\bf Acknowledgement}

The author would like to thank Nick Salter for discussing the content of this paper. She thanks Matt Clay and Dan Margalit for reminding me the fact that ${\cal I}_2$ is a free group. Lastly, she would like to extend her warmest thanks to Benson Farb for his extensive comments as well as for his invaluable support from start to finish.

\section{The proof of Theorem \ref{main}}
Let $g>3$. We assume that the exact sequence (\ref{*}) has a splitting which is denoted by $\phi$ such that $F\circ \phi=id$. The goal of this section is to prove Theorem \ref{main} by contradiction. In all the figures in this section, $*$ represents the puncture of $S_{g,1}$, the genus $g$ surface with one puncture.

\subsection{Background}
In this subsection we discuss some properties of canonical reduction systems and the lantern relation. Let $S=S_{g,p}^b$ be a surface with $b$ boundary components and $p$ punctures. Let Mod$(S)$ (reps. PMod$(S)$) be the \emph{mapping class group} (resp. \emph{pure mapping class group}) of $S$, i.e. the group of isotopy classes of orientation-preserving diffeomorphisms of $S$ fixing the boundary components pointwise and the punctures as a set (resp. pointwise).  By ``simple closed curves", we often mean isotopy class of simple closed curves, e.g. by ``preserve a simple closed curve", we mean preserve the isotopy class of a curve. 

Thurston's classification of elements of Mod$(S)$ is a very powerful tool to study mapping class groups. We call a mapping class $f\in\text{Mod}(S)$ \emph{reducible} if a power of $f$ fixes a nonperipheral simple closed curve. Each nontrivial element $f\in\text{Mod}(S)$ is of exactly one of the following types: periodic, reducible, pseudo-Anosov. See \cite[Chapter 13]{BensonMargalit} and \cite{FLP} for more details. We now give the definition of canonical reduction system. 

\begin{defn}[{\bf Reduction systems}]
 A \emph{reduction system} of a reducible mapping class $h$ in $\text{\normalfont Mod}(S)$ is a set of disjoint nonperipheral curves that $h$ fixes as a set up to isotopy. A reduction system is \emph{maximal} if it is maximal with respect to inclusion of reduction systems for $h$. The \emph{canonical reduction system} $\text{\normalfont CRS}(h)$ is the intersection of all maximal reduction systems of $h$. 
\end{defn}
For a reducible element $f$, there exists $n$ such that $f^n$ fixes each element in CRS$(f)$ and after cutting out CRS$(f)$, the restriction of $f^n$ on each component is either periodic or pseudo-Anosov. See \cite[Corollary 13.3]{BensonMargalit}. Now we mention three properties of the canonical reduction systems that will be used later.
\begin{prop}
$\text{\normalfont CRS}(h^n)$=$\text{\normalfont CRS}(h)$ for any $n$.
\end{prop}
\begin{proof}This is classical; see  \cite[Chapter 13]{BensonMargalit}.
\end{proof}
For a curve $a$ on a surface $S$, denote by $T_a$ the Dehn twist about $a$. For two curves $a,b$ on a surface $S$, let $i(a,b)$ be the geometric intersection number of $a$ and $b$. For two sets of curves $P$ and $T$, we say that $S$ and $T$ \emph{intersect} if there exist $a\in P$ and $b\in T$ such that $i(a,b)\neq 0$. Notice that two sets of curves intersecting does not mean that they have a common element.
\begin{prop}
Let $h$ be a reducible mapping class in $\text{\normalfont Mod}(S)$. If $\{\gamma\}$ and $\text{\normalfont CRS}(h)$ intersect, then no power of $h$ fixes $\gamma$.
\label{noin}
\end{prop}

\begin{proof}
Suppose that $h^n$ fixes $\gamma$. Therefore $\gamma$ belongs to a maximal reduction system $M$. By definition, $\text{CRS}(h)\subset M$. However $\gamma$ intersects some curve in CRS$(f)$; this contradicts the fact that $M$ is a set of disjoint curves. 
\end{proof}
\begin{prop}
Suppose that $h,f\in \text{\normalfont Mod}(S)$ and $fh=hf$. Then $\text{\normalfont CRS}(h)$ and $\text{\normalfont CRS}(f)$ do not intersect.
\label{CRS(x,y)}
\end{prop}
\begin{proof}
By conjugation, we have that CRS($hfh^{-1})=h(\text{CRS}(f))$. Since $hfh^{-1}=f$, we get that CRS$(f)=h(\text{CRS}(f))$. Therefore $h$ fixes the whole set CRS$(f)$. A power of $h$ fixes all curves in CRS$(f)$. By Proposition \ref{noin}, curves in CRS$(h)$ do not intersect curves in CRS$(f)$.
\end{proof}

We denote the symmetric difference of two sets $A$, $B$ by $A\triangle B$.
\begin{lem}
Let $h,f\in \text{\normalfont Mod}(S)$ be two reduced mapping classes such that $hf=hf$. Then $\text{CRS}(h)\triangle \text{CRS}(f)\subset \text{CRS}(hf)$.
\label{CRS(xy)}
\end{lem}
\begin{proof}
Suppose that $\gamma\in$ CRS$(h)$ and $\gamma\notin$ CRS$(f)$. By Corollary \ref{CRS(x,y)}, $\gamma$ does not intersect CRS$(f)$. The canonical form of $f$ has a component $C$ that contains $\gamma$. From $fhf^{-1}=h$ we know that $f$ permutes CRS$(h)$, e.g. a power of $f$ fixes $\gamma$. Since a pseudo-Anosov element does not fix any curve, a power of $f$ is the identity on $C$.

Since $hfh^{-1}=f$, we know that $h$ permutes the components in the canonical form of $f$, e.g. $h(C)$ is another component in the canonical form of $f$. Since $f$ permutes CRS$(h)$, a power of $f$ fixes $\gamma$. This shows that $C$ and $h(C)$ intersect, therefore we have that $h(C)=C$.

Suppose that on the component $C$, the curve $\gamma\notin \text{CRS}(hf)$. This means that there is a curve $\gamma'\subset C$ such that $(hf)^n(\gamma')=\gamma'$ for some integer $n$ and $i(\gamma,\gamma')\neq 0$. A power of $f$ is the identity on $C$, therefore $f$ fixes $\gamma'$. However no power of $h$ fixes $\gamma'$ by Proposition \ref{noin}. Therefore, no power of $hf$ fixes $\gamma'$. This is a contradiction, which shows that $\gamma\in\text{CRS}(hf|_C)$.

For any element $e\in \text{Mod}(S)$ such that the action on $S$ can be broken into actions on components $\{C_1,...,C_k\}$, we have 
\[
\text{CRS}(e|_{C_1})\cup ...\cup \text{CRS}(e|_{C_k})\subset \text{CRS}(e).
\]
Therefore $\gamma\in\text{CRS}(hf|C)\subset \text{CRS}(hf)$.
\end{proof}

Now, we introduce a remarkable relation for $\text{\normalfont Mod}(S)$ that will be used in the proof.
\begin{prop}[\bf\boldmath The lantern relation]
There is an orientation-preserving embedding of $S_{0,4}\subset S$ and let $x,y,z,b_1,b_2,b_3,b_4$ be simple closed curves in $S_{0,4}$ that are arranged as the curves shown in the following figure. 
\begin{figure}[H]
\centering
\includegraphics[scale=0.25]{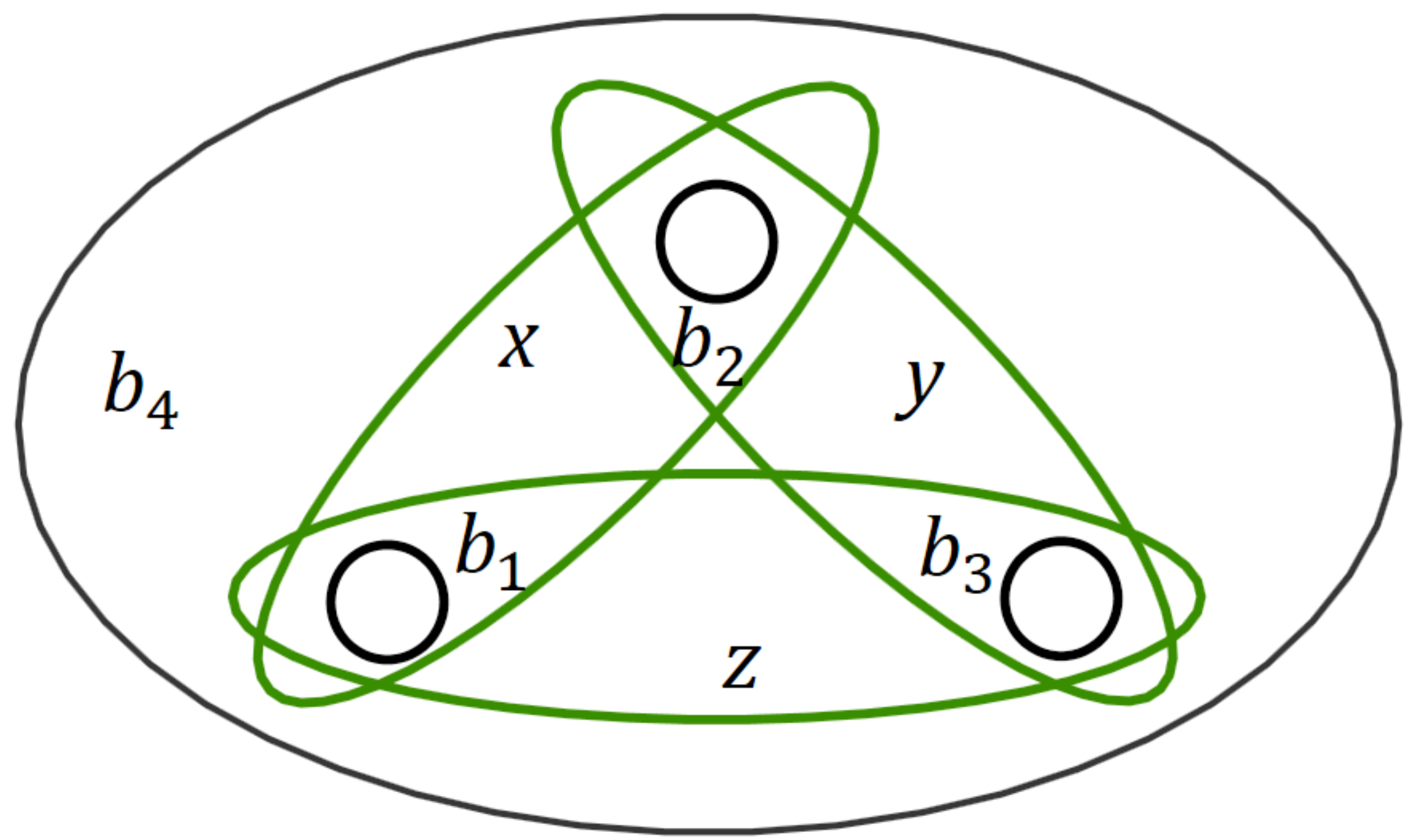}
\end{figure}
In {\normalfont Mod}$(S)$ we have the relation
\[T_xT_yT_z=T_{b_1}T_{b_2}T_{b_3}T_{b_4}.\]
\end{prop}
\begin{proof}
This is classical; see \cite[Chapter 5.1]{BensonMargalit}.
\end{proof}

\subsection{Lifts of bounding pair maps}
Let $\{a,b\}$ be a \emph{bounding pair} as in the following figure, i.e. $a,b$ are nonseparating curves such that $a$ and $b$ bounds a subsurface. Denote by $T_c$ the Dehn twist about a curve $c$. In this subsection, we determine $\phi(T_aT_b^{-1})$. For two curves $c$ and $d$, denote by $i(c,d)$ the geometric intersection number of $c$ and $d$. For a curve $c'$ on $S_{g,1}$, when we say $c'$ is isotopic to a curve $c$ on $S_g$, we mean that $c'$ is isotopic to $c$ on $S_g$.
\begin{figure}[H]
\centering
\includegraphics[scale=0.25]{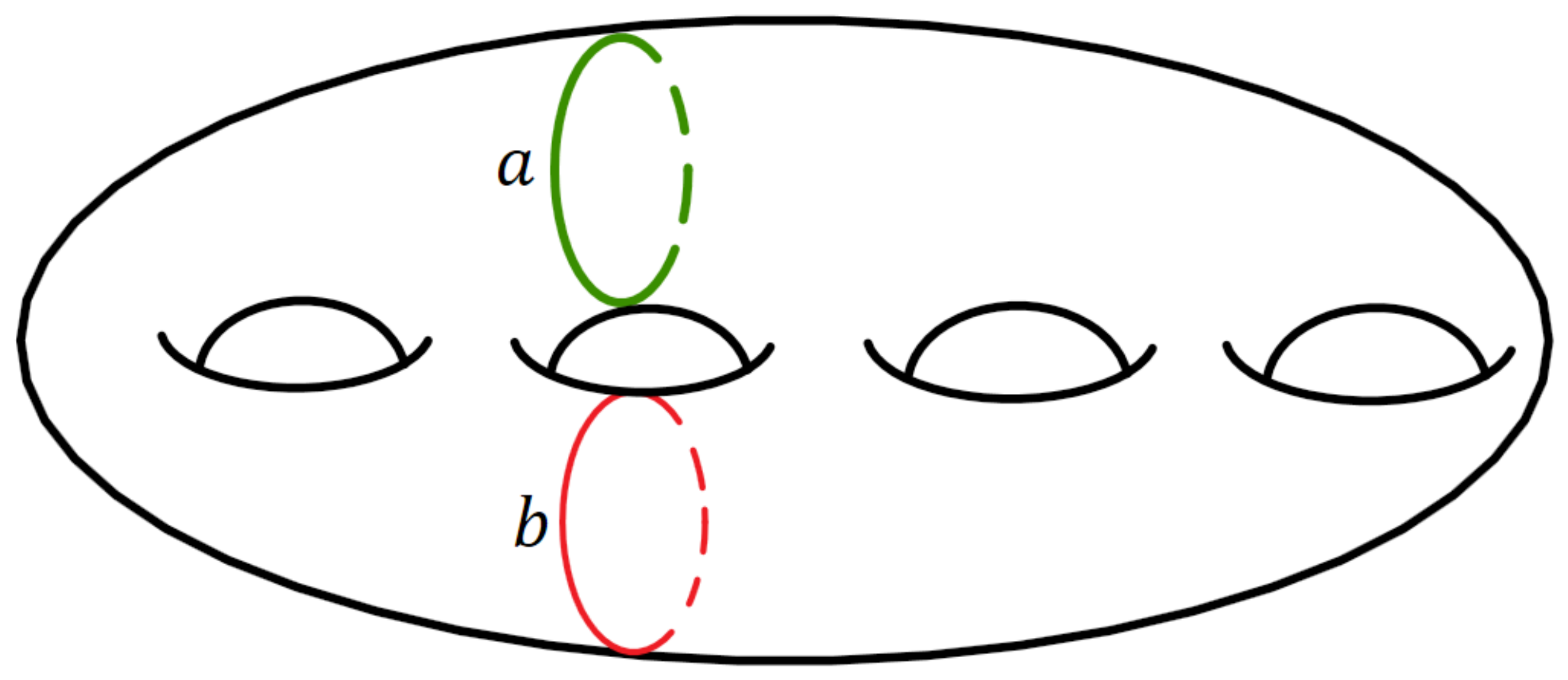}
  \caption{A bounding pair $a,b$}\label{BP}
\end{figure}
\begin{lem}
Let $\{a,b\}$ be a bounding pair as in Figure \ref{BP}. Up to a swap of $a$ and $b$, CRS$(\phi(T_aT_b^{-1}))$ can be one of the following two cases. Moreover, either $\phi(T_aT_b^{-1})=T_{a'}T_{b'}^{-1}$ as in case 1 or there exists an integer $n$ such that $\phi(T_aT_b^{-1})=T_{a'}^{n}T_{a''}^{1-n}T_{b'}^{-1}$ as in case 2. For a Dehn twist $T_s$ about a separating curve $s$, there exists a pair of disjoint curves $s'$ and $s''$ such that they are all isotopic to $s$ and $\phi(T_s)=T_{s'}^{n}T_{s''}^{1-n}$ for some integer $n$.
\begin{figure}[H]
\minipage{0.40\textwidth}
  \includegraphics[width=\linewidth]{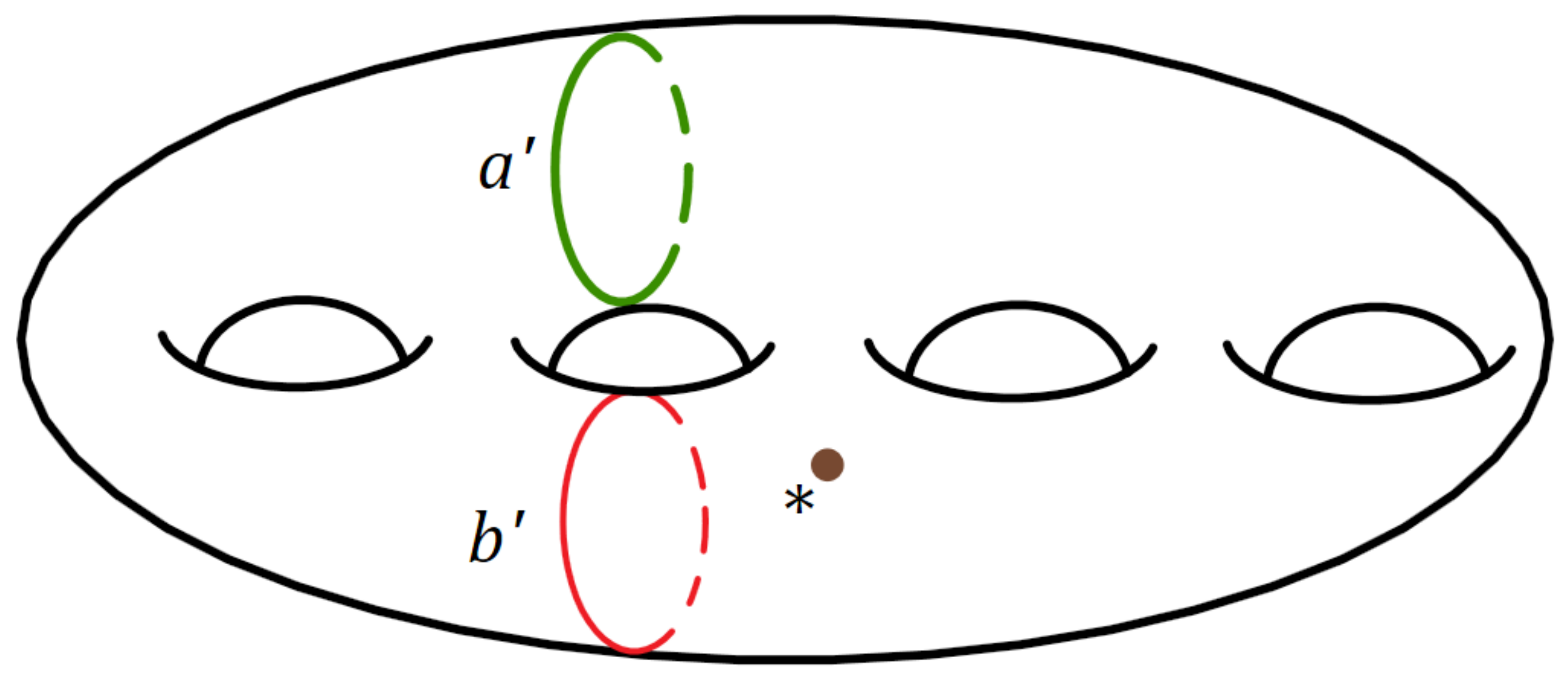}
  \caption{Case 1}
  \label{figure2}
\endminipage\hfill
\minipage{0.40\textwidth}
  \includegraphics[width=\linewidth]{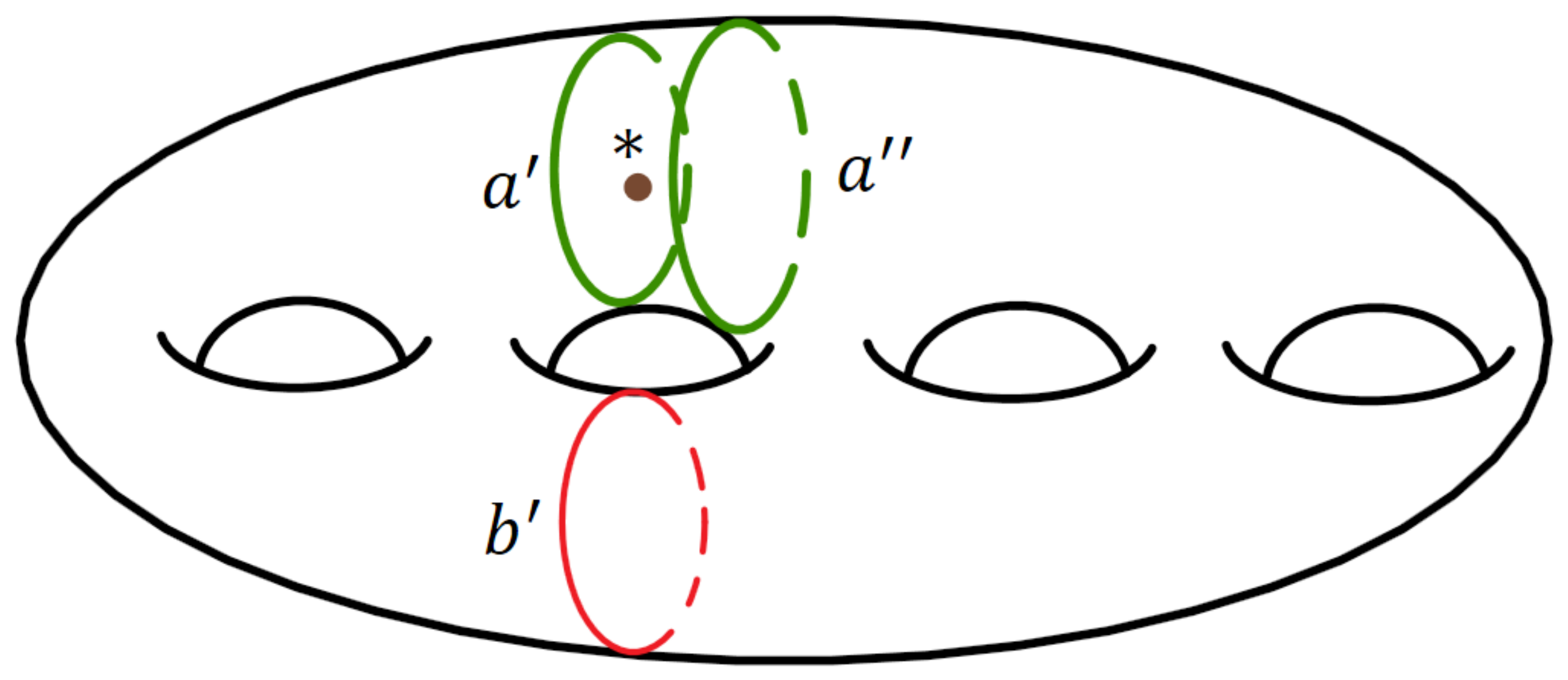}
  \caption{Case 2}\label{figure3}
\endminipage\hfill
\end{figure}
\label{haha}
\end{lem}
\begin{proof}
Let $T_aT_b^{-1}\in {\cal I}_g$ be a bounding pair map. Since the centralizer of $T_aT_b^{-1}\in {\cal I}_{g}$ contains a copy of $\mathbb{Z}^{2g-3}$ as a subgroup of ${\cal I}_g$, the centralizer of $\phi(T_aT_b^{-1})\in {\cal I}_{g,1}$ contains a copy of $\mathbb{Z}^{2g-3}$ as well. However by \cite[Theorem 1]{mccarthy1982normalizers}, the centralizer of a pseudo-Anosov element is virtually cyclic group. $g>3$ implies that $2g-3>3$. Therefore $\phi(T_aT_b^{-1})\in {\cal I}_{g,1}$ is not pseudo-Anosov. For any curve $\gamma'$ on $S_{g,1}$, denote by $\gamma$ the same curve on $S_g$. We decompose the proof into the following three steps.

\begin{claim}[\bf Step 1]
 CRS$(\phi(T_aT_b^{-1}))$  only contains curves that are isotopic to $a$ or $b$.
\end{claim}
\begin{proof}

Suppose the opposite that there exists $\gamma' \in \text{CRS}(\phi(T_aT_b^{-1}))$ such that $\gamma$ is not isotopic to $a$ or $b$. There are two cases.

{\bf Case 1: $\gamma$ intersect $a$ and $b$.} Since a power of $\phi(T_aT_b^{-1})$ fixes $\gamma'$, a power of $T_aT_b^{-1}$ fixes $\gamma$. By CRS$(T_aT_b^{-1})=\{a,b\}$ and Lemma \ref{noin}, we know that $T_aT_b^{-1}$ does not fix $\gamma$. This is a contradiction.

{\bf Case 2: $\gamma$ does not intersect $a$ and $b$.} In this case by the change of coordinate principle, we can always find a separating curve $c$ such that $i(a,c)=0$, $i(b,c)=0$ and $i(c,\gamma)\neq 0$. Since $T_aT_b^{-1}$ and $T_c$ commute in ${\cal I}_{g}$, the two mapping classes $\phi(T_aT_b^{-1})$ and $\phi(T_c)$ commute in ${\cal I}_{g,1}$. This shows that a power of $\phi(T_c)$ fixes CRS($\phi(T_aT_b^{-1})$); more specifically a power of $\phi(T_c)$ fixes $\gamma'$. However by Lemma \ref{noin}, no power of $T_c$ fixes $\gamma$. This is a contradiction.
\end{proof}

\begin{claim}[\bf Step 2]
CRS$(\phi(T_aT_b^{-1}))$ must contain curves that are isotopic to $a$ and $b$.
\end{claim}
\begin{proof}
Suppose the opposite that CRS$(\phi(T_aT_b^{-1}))$ does not contain a curve $\gamma'$ such that $\gamma$ is isotopic to $a$. Then by Step 1, CRS$(\phi(T_aT_b^{-1}))$ either contains one curve $b'$ isotopic to $b$ or two curves $b'$ and $b''$ both isotopic to $b$. After cutting CRS$(\phi(T_aT_b^{-1}))$, we have a component $C$ that is not a punctured annulus. $C$ homeomorphic to the complement of $b$ in $S_g$.

If $\phi(T_aT_b^{-1})$ is pseudo-Anosov on $C$, then the centralizer of $\phi(T_aT_b^{-1})|_C$ at most contains one copy of $\mathbb{Z}$ by \cite[Theorem 1]{mccarthy1982normalizers}. Combining with $T_{b'}$ and $T_{b''}$,  the centralizer of $\phi(T_aT_b^{-1})$ at most contains one copy of $\mathbb{Z}^3$ as a subgroup. This contradicts the fact that the centralizer of $\phi(T_aT_b^{-1})$ contains a subgroup $\mathbb{Z}^{2g-3}$ as a subgroup because $2g-3>3$. Here we need to use $g\ge 4$. Therefore $\phi(T_aT_b^{-1})$ is identity on $C$. This contradicts the fact that $T_aT_b^{-1}$ is not identity on $C$.
\end{proof}

\begin{claim}[\bf Step 3]
Either $\phi(T_aT_b^{-1})=T_{a'}T_{b'}^{-1}$ as in case 1 or there exists an integer $n$ such that $\phi(T_aT_b^{-1})=T_{a'}^{n}T_{a''}^{1-n}T_{b'}^{-1}$ as in case 2. 
\end{claim}
\begin{proof}
Suppose that $\phi(T_aT_b^{-1})$ is pseudo-Anosov on a component $C$ after cutting out CRS$(\phi(T_aT_b^{-1}))$ from $S_{g,1}$. Since $g(C)\ge 1$, there exists a separating curve $s$ on $C$ such that $\phi(T_s)$ commutes with $\phi(T_aT_b^{-1})$. Therefore $\phi(T_aT_b^{-1})$ fixes CRS($\phi(T_s))$, which is either one curve or two curves isotopic to $s$. Thus a power of $\phi(T_aT_b^{-1})$ fixes curves on $C$, which means that $\phi(T_aT_b^{-1})$ is not pseudo-Anosov on $C$. Therefore, $\phi(T_aT_b^{-1})$ is not  pseudo-Anosov on each of the components. By the canonical form of a mapping class, a power of $\phi(T_aT_b^{-1})$ is a product of Dehn twists about CRS$(\phi(T_aT_b^{-1}))$. By the fact that $\phi(T_aT_b^{-1})$ is a lift of $T_aT_b^{-1}$, the lemma holds.
\end{proof}
The same argument works for $T_s$ the Dehn twist about a separating curve $s$.
\end{proof}

When $n=0$ or $n=1$, we have that $(T_{a'})^n(T_{a''})^{1-n}T_{b'}=T_{a''}'T_{b'}^{-1}$ or $(T_{a'})^n(T_{a''})^{1-n}T_{b'}=T_{a'}'T_{b'}^{-1}$. Therefore, we can combine the results to get that $\phi(T_aT_b^{-1})=(T_{a'})^n(T_{a''})^{1-n}T_{b'}$. In $\phi(T_aT_b^{-1})=(T_{a'})^n(T_{a''})^{1-n}T_{b'}$, denote $(T_{a'})^n(T_{a''})^{1-n}$ by the \emph{$a$ component of $\phi(T_aT_b^{-1})$}. Notice that by symmetry, the $b$ component of $\phi(T_aT_b^{-1})$ could also be a product of Dehn twists. In the following lemma, we will prove that the $a$ component of $\phi(T_aT_b^{-1})$ does not depend on the choice of $b$.

\begin{lem}
For two bounding pairs $\{a,b\}$ and $\{a,c\}$, the $a$ component of $\phi(T_aT_b^{-1})$ is the same as the $a$ component of $\phi(T_aT_c^{-1})$.
\label{depend}
 \end{lem}
\begin{proof}
If $b,c$ are disjoint, $\phi(T_aT_b^{-1})$ and $\phi(T_aT_c^{-1})$ commute. By Lemma \ref{CRS(xy)}, we have 
\[
\text{CRS}(\phi(T_aT_b^{-1}))\triangle \text{CRS}(\phi(T_aT_c^{-1})^{-1})\subset\text{CRS}(\phi(T_aT_b^{-1}))\phi(T_aT_c^{-1})^{-1}).
\]
 However $\phi(T_aT_b^{-1}))\phi(T_aT_c^{-1})^{-1}=\phi(T_cT_b^{-1})$, we have that $\text{CRS}(\phi(T_aT_b^{-1}))\phi(T_aT_c^{-1})^{-1})$ only contains curves that are isotopic to $b$ or $c$. So $\text{CRS}(\phi(T_aT_b^{-1}))\triangle \text{CRS}(\phi(T_aT_c^{-1})^{-1})$ does not contain curves isotopic to $a$. This shows that the $a$ components of $\phi(T_aT_b^{-1})$ and $\phi(T_aT_c^{-1})$ are the same so that they can cancel each other through multiplication.

When $b,c$ intersect, there are a series of curves $\{b_1=b,b_2,...,b_n=c\}$ such that $i(b_i, b_{i+1})=0$ and $i(a,b_i)=0$ for all $1 \le i \le n-1$. This fact can be deduced from the connectivity of the complex of homologous curves, e.g. see \cite{putman2006note}. Therefore, the $a$ components of $\phi(T_aT_b^{-1})$ and $\phi(T_aT_c^{-1})$ are the same.

\end{proof}

We denote by the capital letter $A$ the subset of curves in CRS$(\phi(T_aT_b^{-1}))$ that are isotopic to $a$. By Lemma \ref{depend}, $A$ only depends on the curve $a$. It can be a one-element set or a two-element set.

\begin{lem}
If $i(a,b)=0$, then $A$ is disjoint from $B$.
\end{lem}
\begin{proof}
Suppose that $a,b$ are nonseparating. The case of separating curves are the same. If $a,b$ bound, then by Lemma \ref{haha}, $A$ and $B$ are disjoint. If $a,b$ do not bound, then there are curves $c,d$ such that they form the following configuration. 
\begin{figure}[H]
\centering
\includegraphics[scale=0.25]{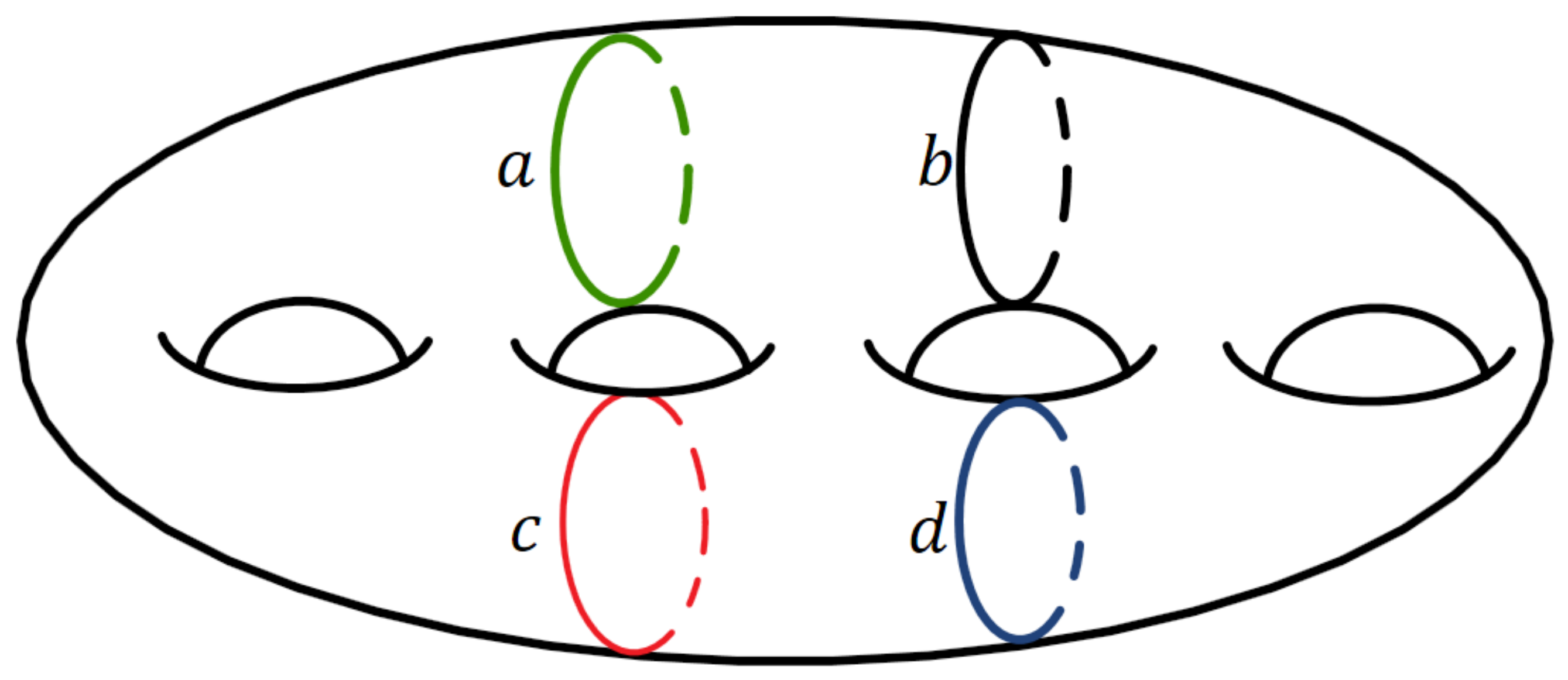}
\end{figure}
$g\ge 4$ is needed here. Since $\phi(T_aT_c^{-1})$ and $\phi(T_bT_d^{-1})$ commute, their canonical reduction systems do not intersect by Corollary \ref{CRS(x,y)}. Therefore $A$ and $B$ are disjoint.
\end{proof}

\subsection{A nonsplitting lemma for the braid group}
Let $D_n$ be an $n$-punctured 2-disk. The n-strand \emph{pure braid group} is denoted by $PB_n$, i.e. the pure mapping class group of $D_n$ fixing the $n$ punctures pointwise. In this subsection, we prove a nonsplitting lemma for the braid group that will be used in the proof of Theorem \ref{main}.

\begin{lem}
Let ${\cal F}: PB_4\to PB_3$ be the forgetful map forgetting the 4th punctures. There is no homomorphism ${\cal G}:PB_3\to PB_4$ such that Dehn twists map to Dehn twists, the center maps to the center and ${\cal F}\circ {\cal G}=id$.
\label{nonsplitting}
\end{lem}
\begin{proof}
Suppose the opposite that we have ${\cal G}:PB_3\to PB_4$ such that Dehn twists map to Dehn twists, the center maps to the center and ${\cal F}\circ {\cal G}=id$. Let $c$ be a simple closed curve on $D_3$ and we call $c'$ the lift on $D_4$ such that ${\cal G}(T_c)=T_{c'}$. In the figure below, the lantern relation gives $T_aT_bT_c=T_d\in PB_3$. Therefore we have $T_{a'}T_{b'}T_{c'}=T_{d'}\in PB_4$. Because ${\cal G}$ maps the center to the center, $d'$ is the boundary curve of $D_n$. Since $a,b,c$ do not intersect $d$, we have that $T_a,T_b,T_c$ commute with $T_d$.

\begin{center}
\includegraphics[scale=0.25]{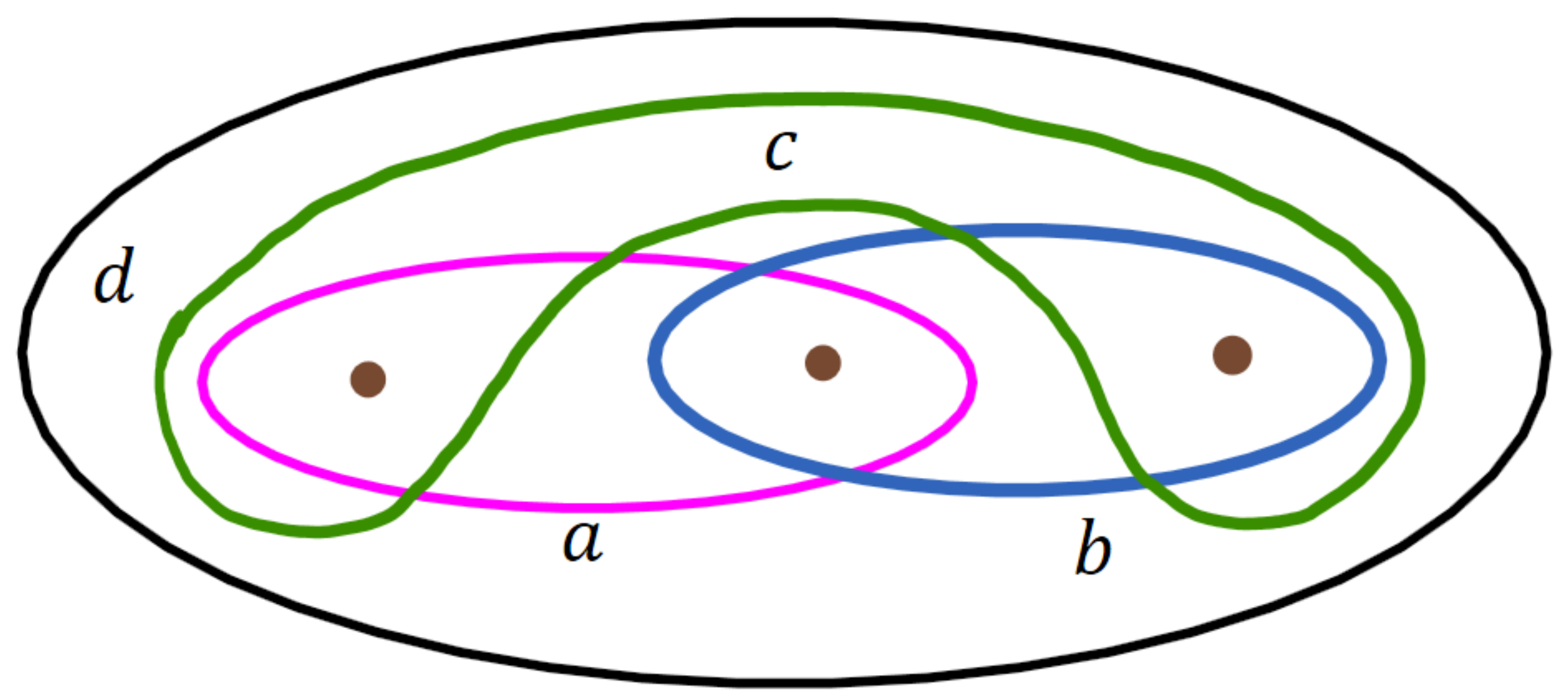}
\end{center}
If $i(a',b')>2$, then $T_{a'}T_{b'}$ is pseudo-Anosov on some subspace of $D_4$ by Thurston's construction, e.g. see  \cite[Proposition 2.13]{chen4}. Therefore, $T_{a'}T_{b'}=T_{d'}T_{c''}^{-1}$ is not a multitwist. So $i(a',b')=2$. Every curve in $D_n$ surrounds several points. For example, $a\subset D_3$  surrounds 2 points. There are several cases we need to concern about the number of surrounding points.

\begin{figure}[H]
\minipage{0.32\textwidth}
\includegraphics[width=\linewidth]{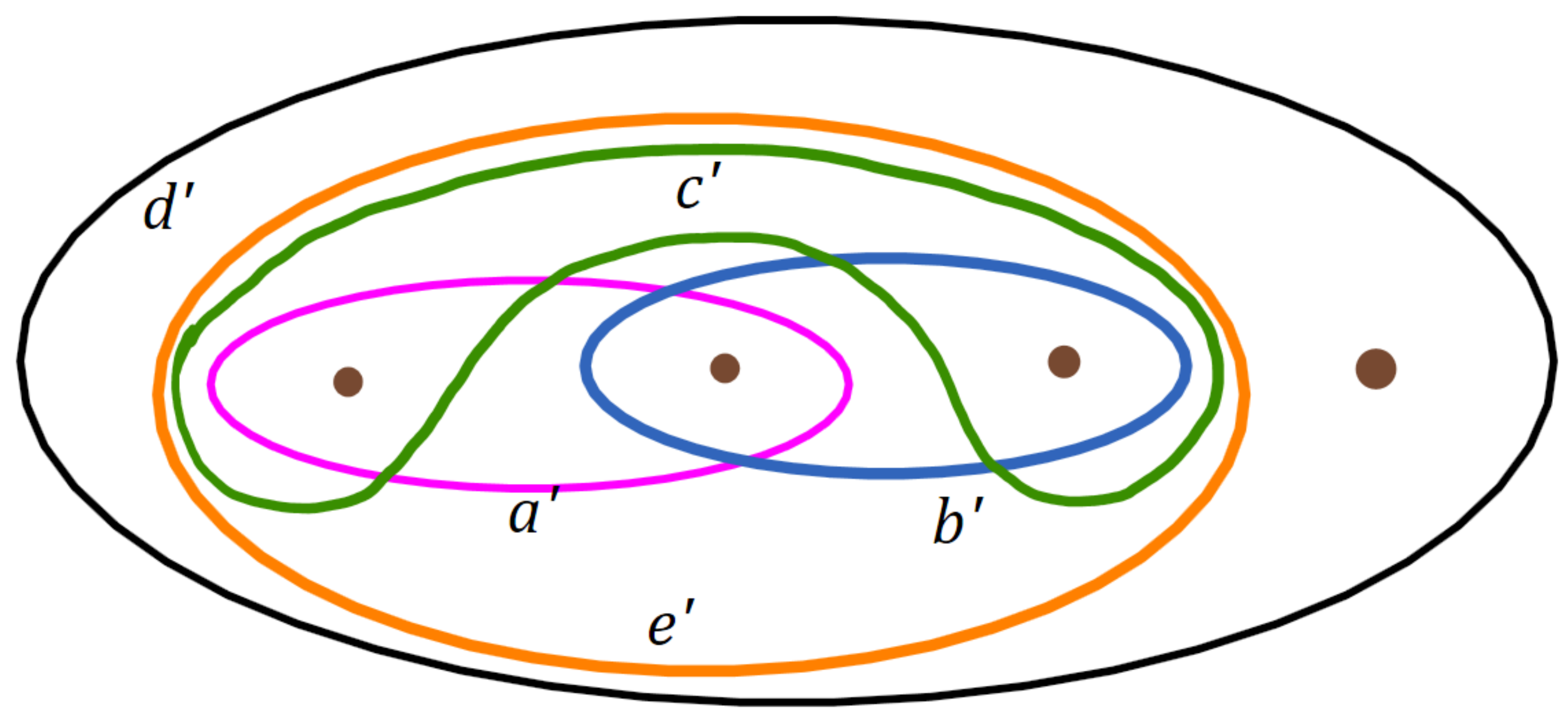}
\caption{Case 1}\label{case1}
\endminipage\hfill
\minipage{0.32\textwidth}
  \includegraphics[width=\linewidth]{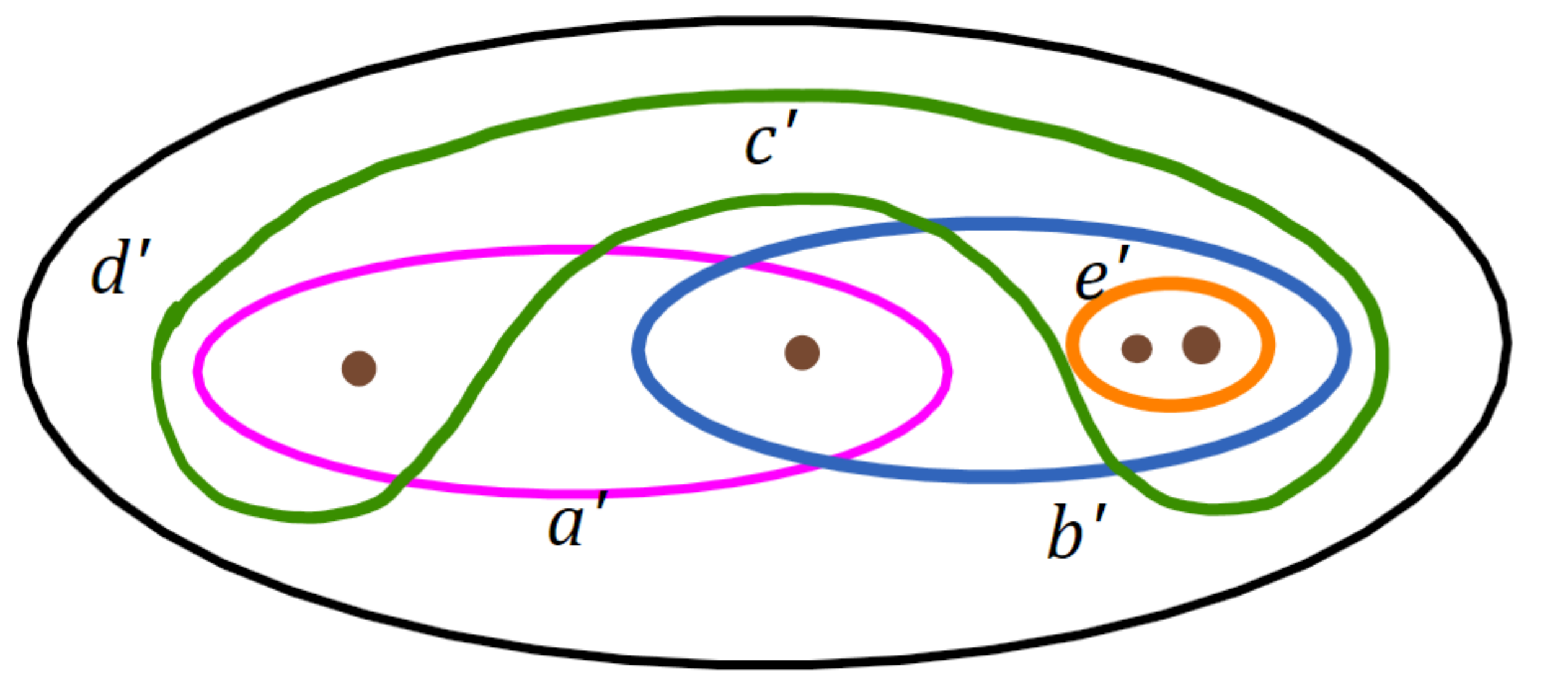}
  \caption{Case 2}\label{case2}
\endminipage\hfill
\minipage{0.32\textwidth}
  \includegraphics[width=\linewidth]{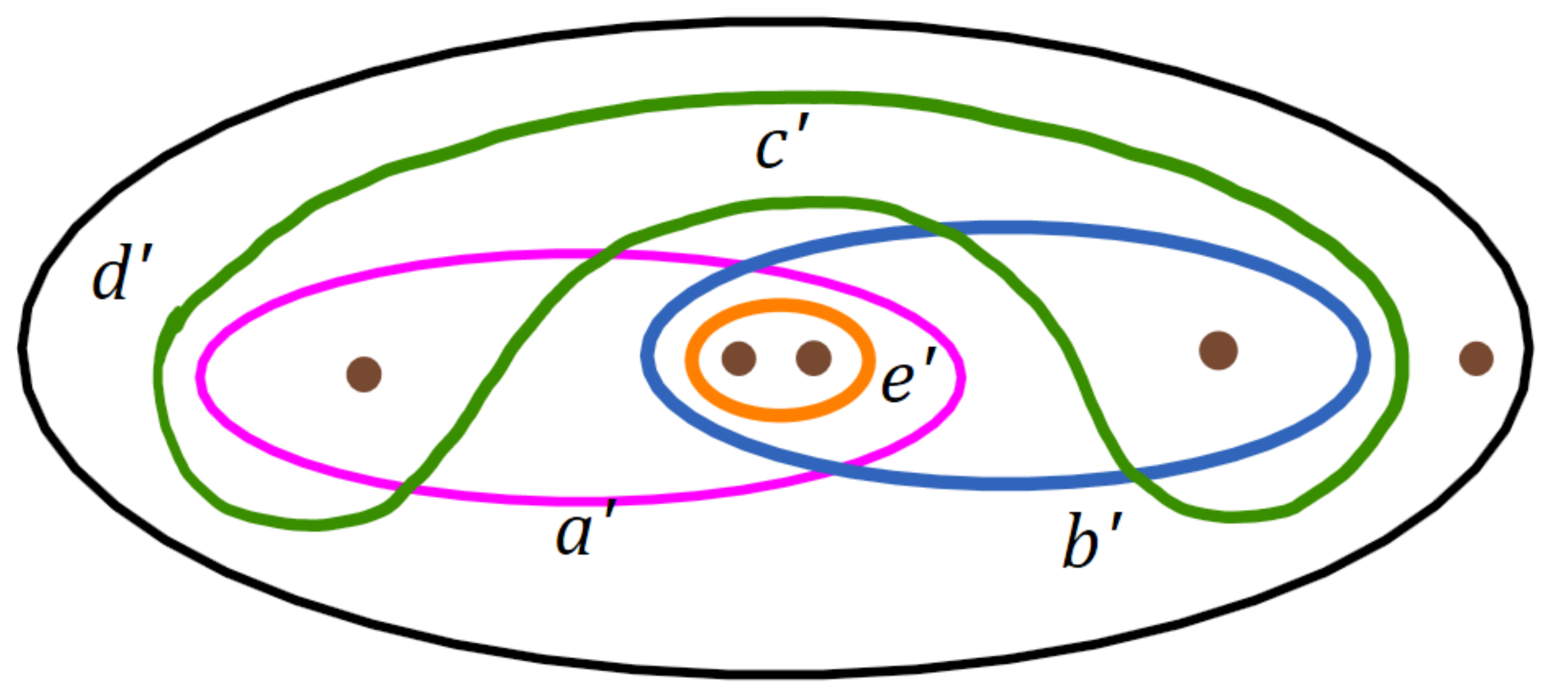}
    \caption{Case 3}\label{case3}
\endminipage\hfill
\end{figure}

{\bf Case 1: $a'$ bounds 2 points and $b'$ bounds 2 points.} Then we have $T_{a'}T_{b'}=T_{e'}T_{c'}^{-1}$ by the lantern relation as is shown in Figure \ref{case1}. We also have the relation $T_{a'}T_{b'}=T_{d'}T_{c'}^{-1}$ from the lift of the relation $T_aT_bT_c=T_d\in PB_3$. However CRS$(T_{d'}T_{c'}^{-1})=\{c'\}$ $\neq$ CRS$(T_{e'}T_{c'}^{-1})=\{e',c'\}$. This is a contradiction.\\

{\bf Case 2: $a'$ bounds 2 points and $b'$ bounds 3 points.} Then we have $T_{a'}T_{b'}=T_{d'}T_{e'}T_{c'}^{-1}$ by the lantern relation as is shown in Figure \ref{case2}. However $\#$CRS$(T_{d'}T_{e'}T_{c'}^{-1})=2>1=\#$CRS$(T_{d'}T_{c'}^{-1})$. This is a contradiction.\\

{\bf Case 3: $a'$ bounds 3 points and $b'$ bounds 3 points.} Then we have $T_{a'}T_{b'}=T_{d'}T_{e'}T_{c'}^{-1}$ by the lantern relation as is shown in Figure \ref{case3}. We have $\#$CRS$(T_{d'}T_{e'}T_{c'}^{-1})=2>1=\#$CRS$(T_{d'}T_{c'}^{-1})$. This is a contradiction.

\end{proof}

\subsection{Proof of Theorem \ref{main}}
In this proof, we do a case study on the possibilities of $\phi(T_aT_e^{-1})$ for a bounding pair map $T_aT_e^{-1}$. Case 1 is when the $a$ component is not a single Dehn twist. We reach a contradiction by the lantern relation. Case 2 is when the component of every curve is a single Dehn twist, we use Lemma \ref{nonsplitting} to cause contradiction.
\begin{proof}[\bf Proof of Theorem \ref{main}] 

We break our discussion into the following two cases.\\

{\bf\boldmath Case 1: there is a bounding pair map $T_aT_e^{-1}$ such that  
\[
\phi(T_aT_e^{-1})=(T_{a'})^n(T_{a''})^{1-n}T_{e'}^{-1}\text{  where $n\neq 0,1$.} \]
}
There exist curves $b,c,d$ such that $a,b,c,d$ form a 4-boundary disk as Figure \ref{figureA}. We need $g>3$ here.
\begin{figure}[H]
\minipage{0.40\textwidth}
  \includegraphics[width=\linewidth]{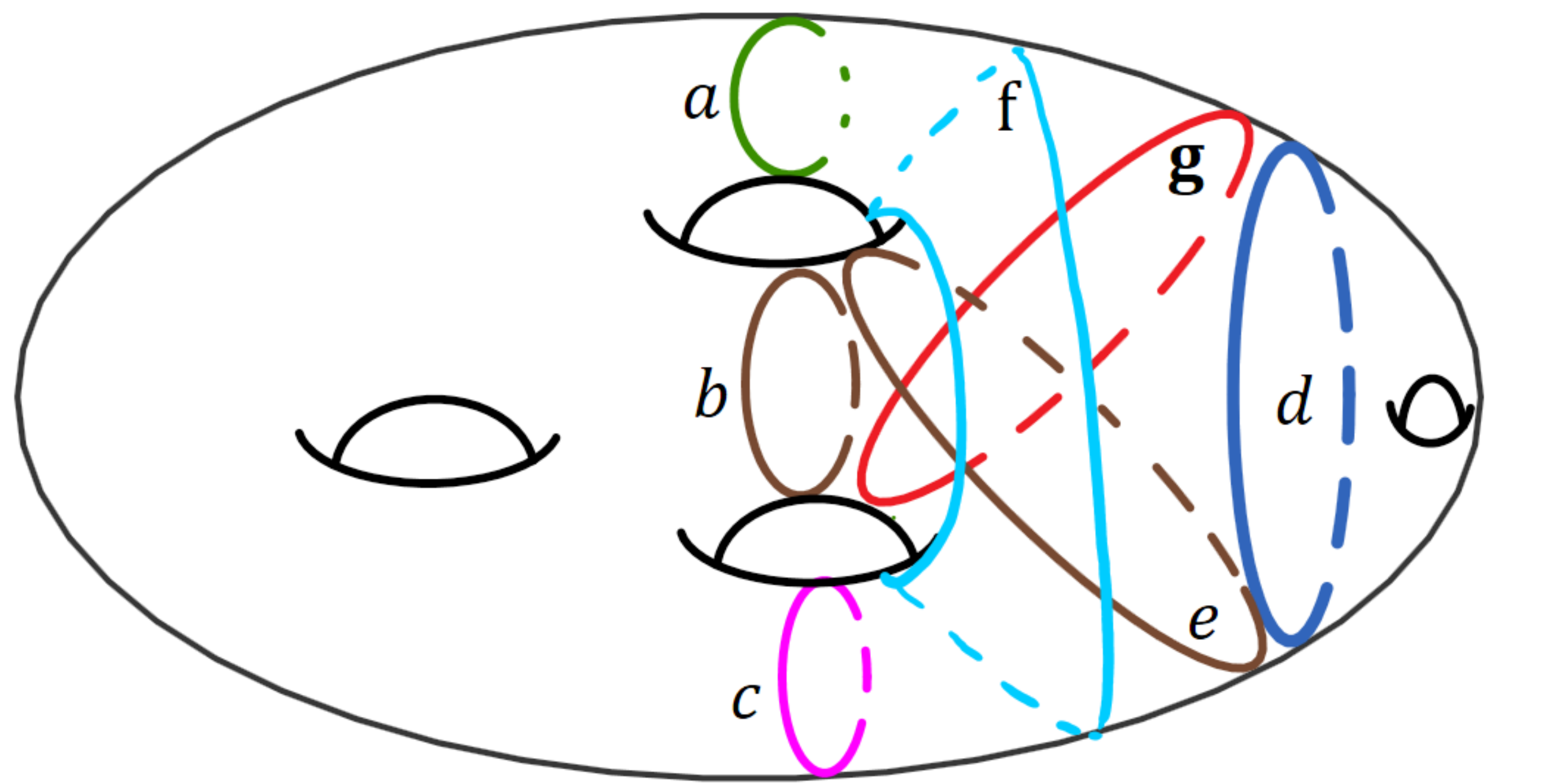}
  \caption{On $S_g$}
  \label{figureA}
\endminipage\hfill
\minipage{0.40\textwidth}
  \includegraphics[width=\linewidth]{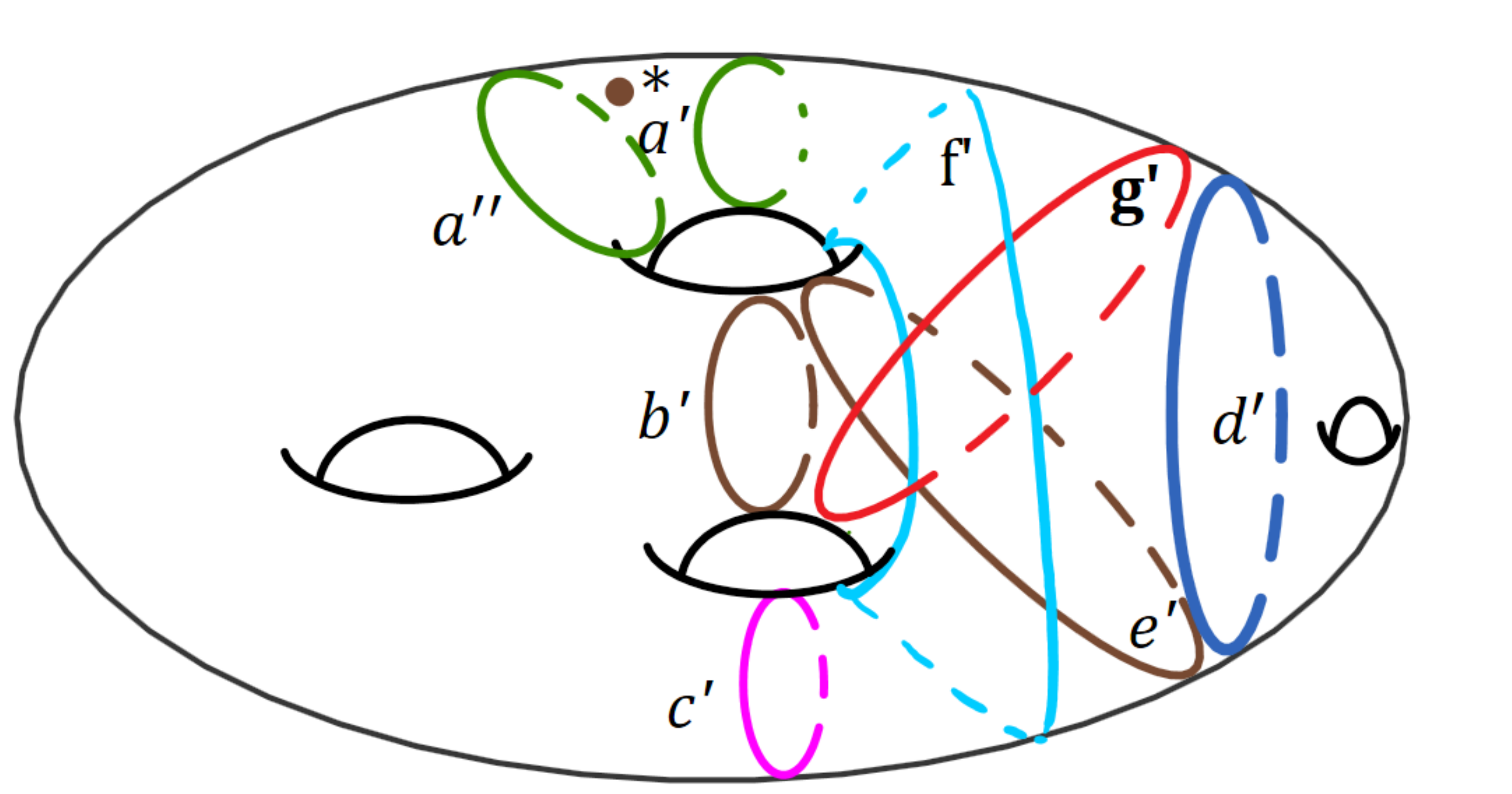}
  \caption{On $S_{g,1}$}
  \label{figureBe}
\endminipage\hfill
\end{figure}

There are two more curves $f,g$ such that we have the lantern relation $T_aT_e^{-1}T_bT_f^{-1}T_cT_g^{-1}T_d=1$. The lifts $\{b',c',d',f',g'\}$ of  $\{b,c,d,f,g\}$ do not intersect $a',a''$ as in Figure \ref{figureBe}. After applying $\phi$, we have 
\[
(T_{a'})^n(T_{a''})^{1-n}T_{e'}^{-1}T_{b'}T_{f'}^{-1}T_{c'}T_g^{-1}T_{d'}=1.
\] Since $T_{a'}T_{e'}^{-1}T_{b'}T_{f'}^{-1}T_{c'}T_{g'}^{-1}T_{d'}=1$ by the lantern relation on $S_{g,1}$, we have $(T_{a'})^n(T_{a''})^{1-n}=T_{a'}$. It means that $n=1$, which contradicts our assumption on $n$. This proof also works for the Dehn twist $T_s$ about a  separating curve $s$.\\

{\bf\boldmath Case 2:  for any bounding pair map $T_aT_e^{-1}$, we have $\phi(T_aT_e^{-1})=T_{a'}T_{e'}^{-1}$ and for any Dehn twist $T_s$ about a separating curve $s$, we have $\phi(T_s)=T_{s'}$}

Let $S_{g,p}^b$ be a genus $g$ surface with $p$ punctures and $b$ boundary components. In this case, firstly we want to locate $*$. Let us decompose the surface into pair of pants as the following Figure \ref{figure5}. 

\begin{figure}[H]
\minipage{0.32\textwidth}
\includegraphics[width=\linewidth]{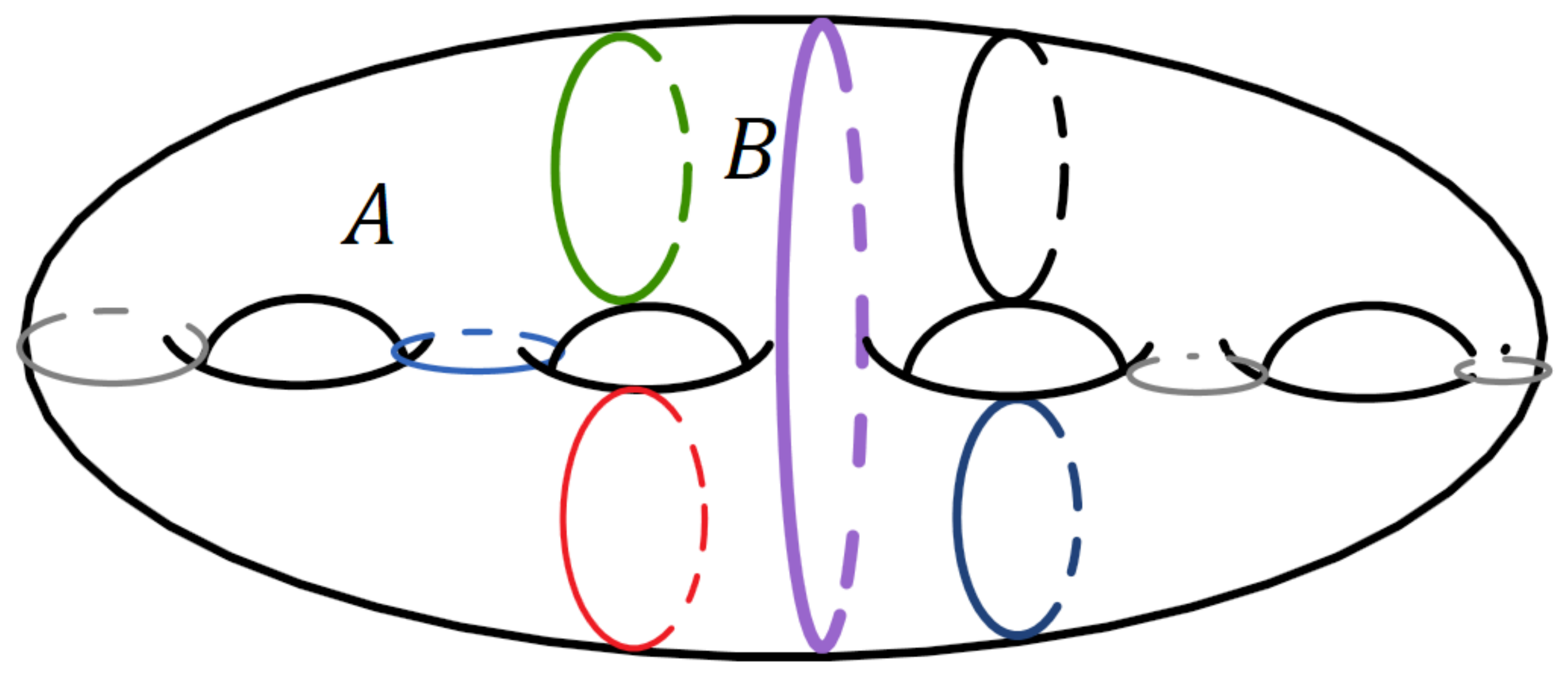}
\caption{A decomposition}
  \label{figure5}
\endminipage\hfill
\minipage{0.32\textwidth}
  \includegraphics[width=\linewidth]{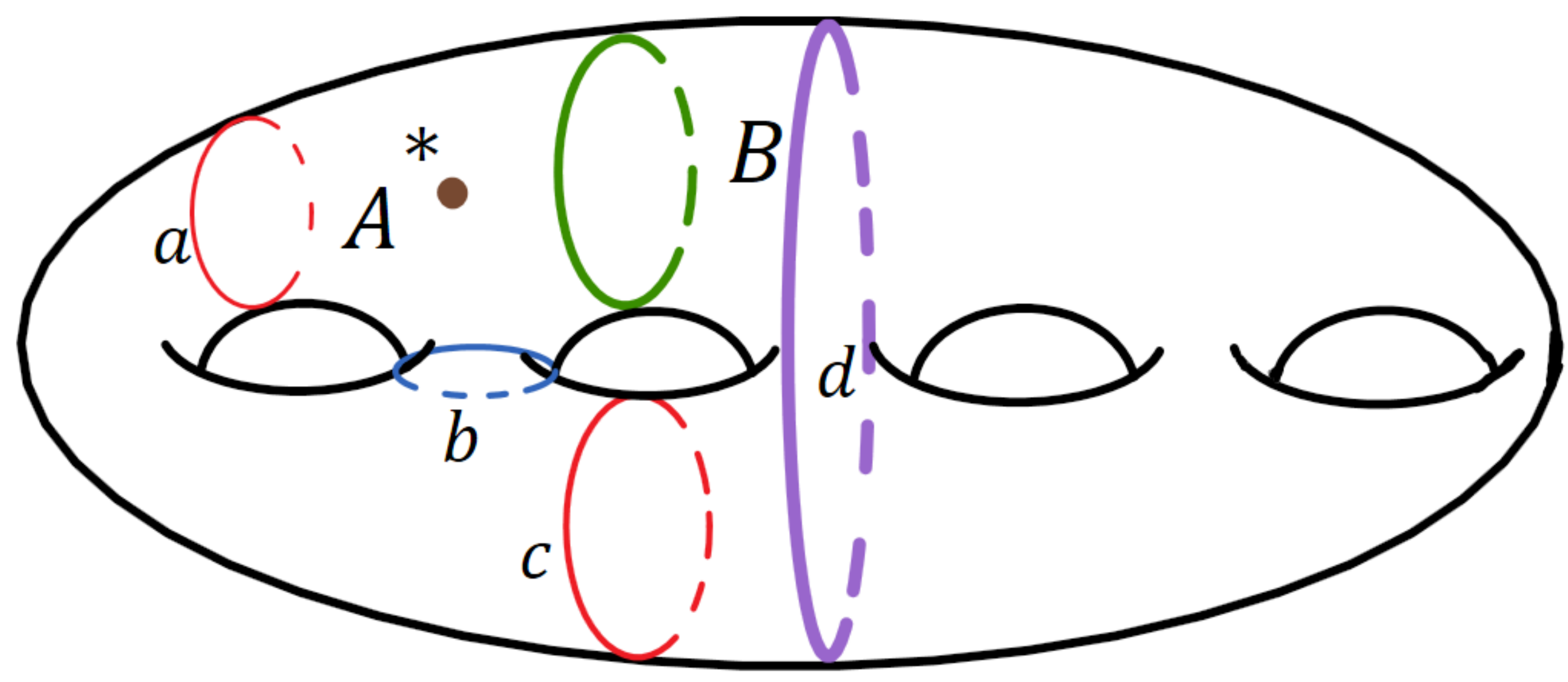}
  \caption{possibility 1}
  \label{figure6}
\endminipage\hfill
\minipage{0.32\textwidth}
  \includegraphics[width=\linewidth]{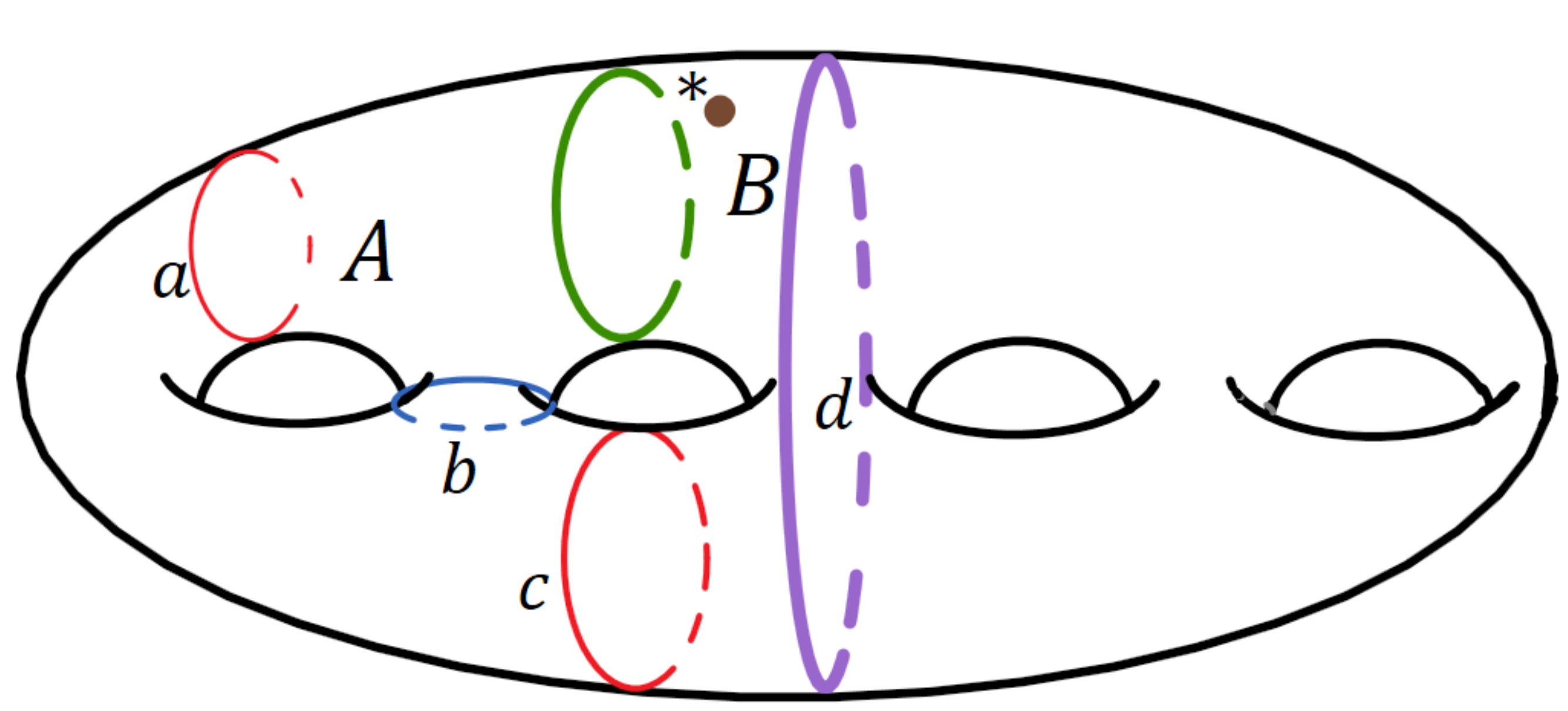}
    \caption{possibility 2}
\label{figure7}
\endminipage\hfill
\end{figure}
The location of $*$ can be either in a pair of pants where all three curves are nonseparating like $A$ or one of them is separating like $B$. Suppose without loss of generality that $*$ lands on $A$ or $B$. If $*$ lands on $A$, we use Figure \ref{figure6} to find four curves $a,b,c,d$ and if $*$ lands on $B$, we use Figure \ref{figure7} to find four curves $a,b,c,d$. The curves $a,b,c,d$ that we find satisfy the following properties:\\
1) $d$ is separating and $a\cup b\cup c\cup d$ bounds a 4-boundary sphere $S\approx S_0^4\subset S_g$. \\
2) The lifts $a',b',c',d'$ are 4 disjoint simple closed curves on $S_{g,1}$ such that $d'$ is separating and $a'\cup b'\cup c'\cup d'$ bounds a 4-boundary sphere with $*$ in $S'\approx S_{0,1}^4\subset S_{g,1}$. See the following figures.

\begin{figure}[H]
\minipage{0.40\textwidth}
  \includegraphics[width=\linewidth]{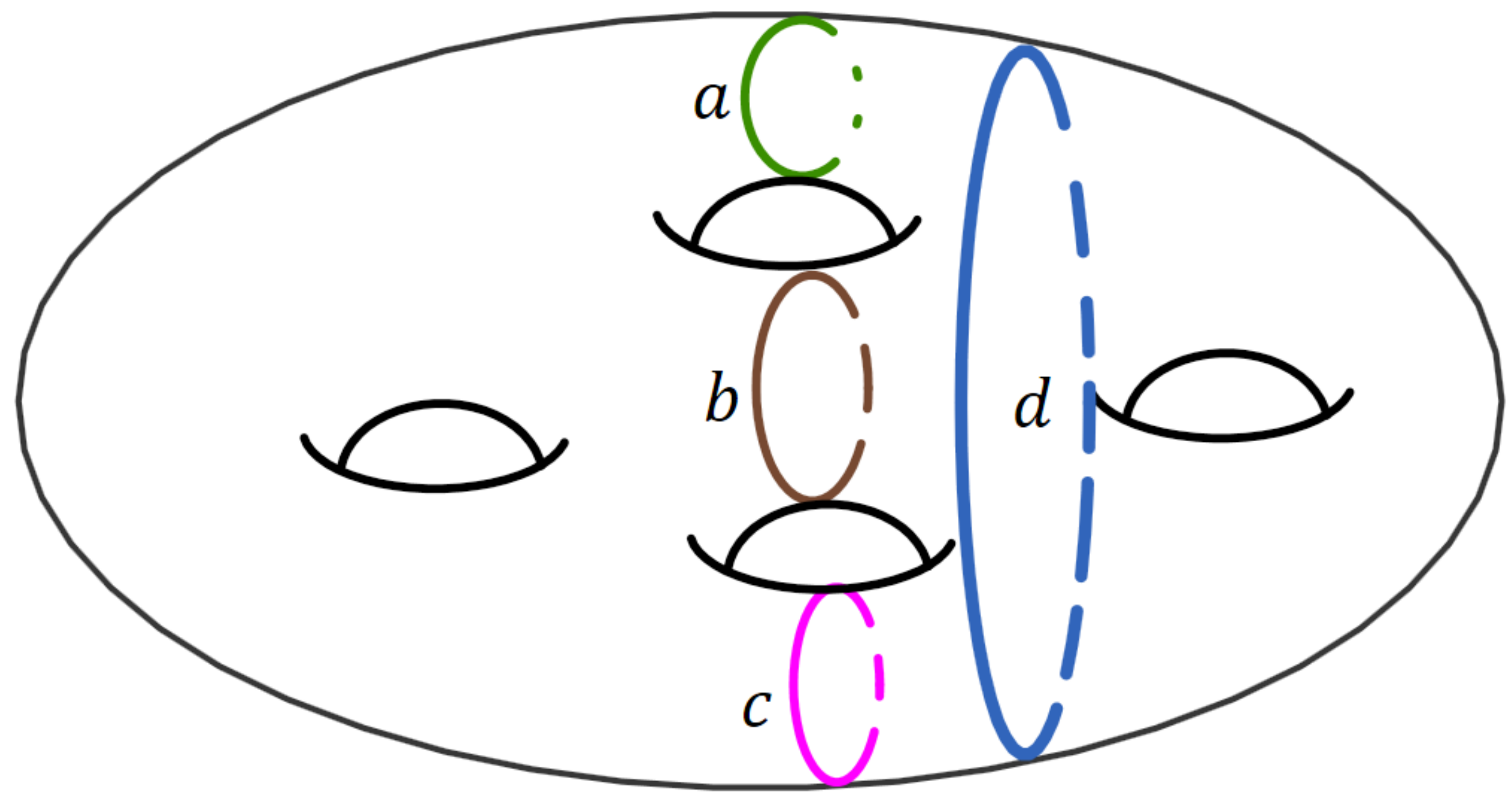}
  \caption{On $S_g$}
\endminipage\hfill
\minipage{0.40\textwidth}
  \includegraphics[width=\linewidth]{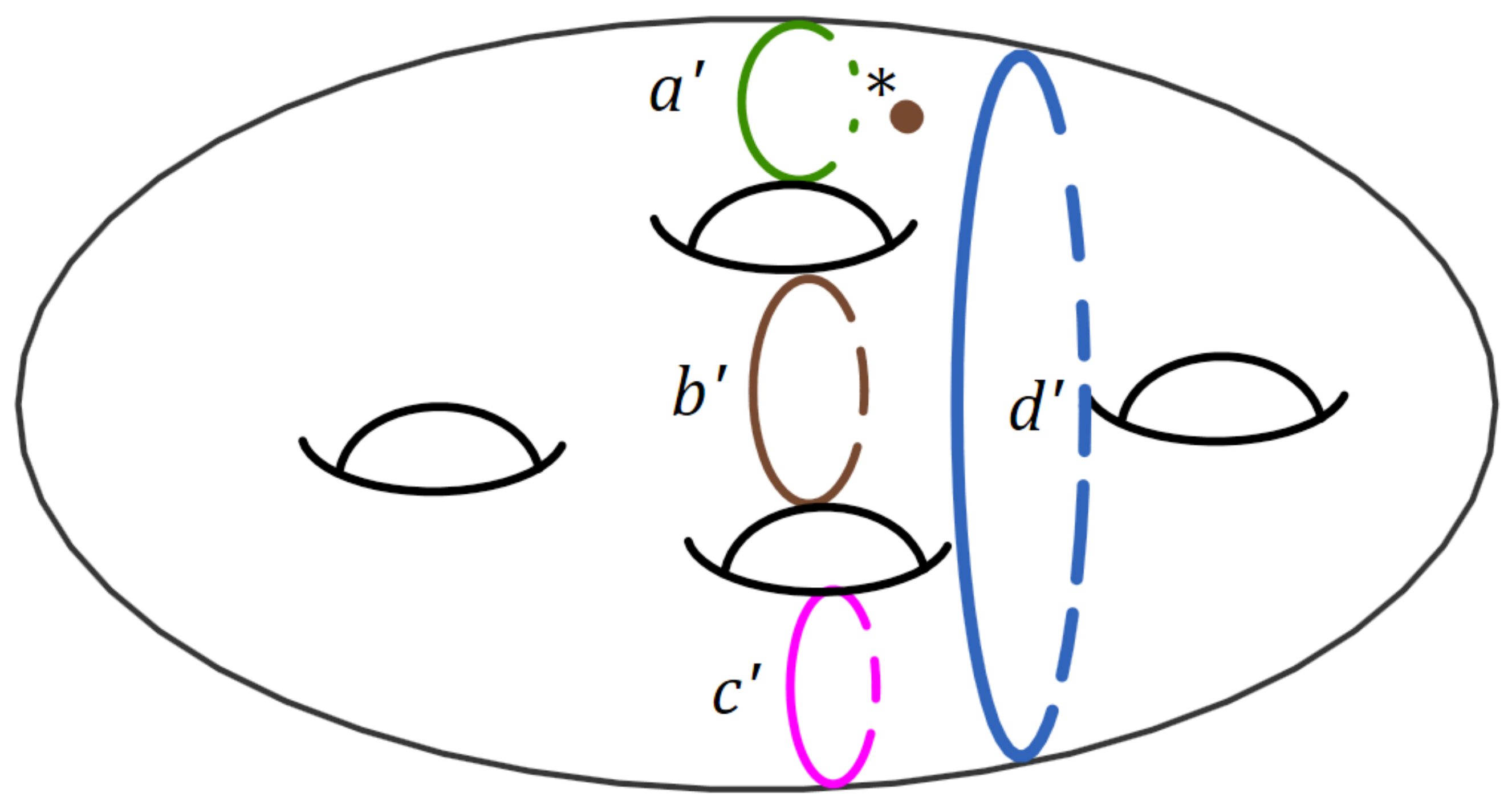}
  \caption{On $S_{g,1}$}
\endminipage\hfill
\end{figure}

\begin{claim}Let $W$ be the subgroup of ${\cal I}_g$ generated by bounding pair maps with curves on $S$. Let $W'$ be the subgroup of ${\cal I}_{g,1}$ generated by bounding pair maps with curves on $S'$ such that one of the curves lies in $a',b',c'$. We have that
\[ W\cong PB_3 \text{  and   }
W'\cong PB_4.\]
\end{claim}
\begin{proof}
W only acts nontrivially on $S\cong S_0^4$. After gluing punctured disks to the boundaries $a,b$ and $c$, there is a homomorphism $\mu: W\to PB_3$. Since every closed curve inside $S_{0,3}^1$ is isotopic to one of the boundary components, every bounding pair map of $W$ maps to a Dehn twist in $PB_3$ under $f$. It is clear that $\phi$ is surjective. If $f\in \text{ker}(\mu)$ as a mapping class on $S$, then $f$ is either trivial or equal to a product of Dehn twists on $a,b,c$. However, we claim that a nontrivial product of Dehn twists on $a,b,c$ is not in Torelli group, which shows that $\mu$ is injective. Suppose the opposite that $f=T_a^{m}T_b^nT_c^l\in {\cal I}_g$ and $l\neq 0$. Let $a',b'$ be two curves and denote by $I(a',b')$ the \emph{algebraic intersection number} of $a'$ and $b'$. For $x\in H_1(S_g;\mathbb{Z})$, we have that $T_a^{m}T_b^nT_c^l(x)=mI(a,x)a+nI(b,x)b+lI(c,x)c+x$. The fact that $T_a^{m}T_b^nT_c^l$ is in ${\cal I}_g$ implies that $mI(a,x)a+nI(b,x)b+lI(c,x)c=0$ for any $x$. Since $a,b$ are independent, there exists an element $x$ such that  $I(x,a)=0$ and $I(x,b)=1$. Since $a\cup b\cup c$ separate implying that $a+b+c=0$, we have that $I(c,x)=-1$. This contradicts $mI(a,x)a+nI(b,x)b+lI(c,x)c=nb-lc=0$ because $b,c$ are independent and $l\neq 0$. For the same reason $W'\cong PB_4$.
\end{proof}

The lifts of elements in $W$ is inside $W'$ and $F:W'\to W$ is the forgetful map forgetting the puncture ${\cal F}:PB_4\to PB_3$. To conclude the proof of the theorem we only need to apply Lemma \ref{nonsplitting} that there is no splitting of ${\cal F}$ satisfying our assumption.

\end{proof}

\section{Torelli spaces with punctures}

In this section, we discuss the ``section problem" for the universal Torelli bundle with punctures.

\subsection{Translation to a group theoretical problem}
We first translate the ``section problem" of the universal Torelli surface bundle into a group-theoretic statement. As is discussed in \cite[Chapter 2.1]{lei1}, we have the following correspondence when $g>1$:
\newcommand{\pctext}[2]{\text{\parbox{#1}{\centering #2}}}
\begin{alignat}{2}
      \label{correspondence}
 \Biggl\{ \pctext{1.5in}{Conjugacy classes of representations $\rho: \pi_1(B)\to \text{Mod}_g$ }\Biggr\}                                                      
      \Longleftrightarrow \Biggl\{  \pctext{1.5in}{Isomorphism classes of oriented $S_g$-bundles over $B$} \Biggr\}.    
      \end{alignat}
Let $f:E\to B$ be a surface bundle determined by $\rho: \pi_1(B)\to \text{Mod}_g$. Let $f_*:\pi_1(E)\to \pi_1(B)$ be the map on the fundamental groups. By the property of pullback diagrams, finding a splitting of $f_*$ is the same as finding a homomorphism $p$ that makes the following diagram commute, i.e. $\pi_{g,1}\circ p=\rho$.
\begin{equation}
\xymatrix{
\pi_1(E)\ar[r]\ar[d]^{f_*}& \text{Mod}_{g,1}\ar[d]^{\pi_{g,1}}\\
\pi_1(B)\ar[r]^{\rho}\ar[ru]^{p} & \text{Mod}_g.}
\label{CDD}
\end{equation}

We have the following correspondence:
\begin{alignat}{2}
      \label{correspondence2}
 \Biggl\{ \pctext{2in}{Homotopy classes of continuous sections of $S_g\to E\xrightarrow{f} B$}\Biggr\}                                                      
      \Longleftrightarrow \Biggl\{  \pctext{3in}{Homomorphisms $p$ satisfying diagram (\ref{CDD}) up to conjugacy by an element in Ker$(\pi_{g,1})\cong \pi_1(S_g)$} \Biggr\}.    
      \end{alignat}

By the correspondence (\ref{correspondence2}), we can translate Theorem \ref{Torelli} into the following group-theoretic statement. Let ${\cal PI}_{g,n} \xrightarrow{{T\pi_{g,n}}} {\cal I}_g$ and ${\cal I}_{g,n} \xrightarrow{{T\pi'_{g,n}}} {\cal I}_g$ be the forgetful maps forgetting the punctures. Let ${\cal I}_{g,n}\xrightarrow{Tp_{g,n,i}} \text{Mod}_{g,1}$ be the forgetful homomorphism forgetting the fixed points $\{x_1,...,\hat{x_i},...,x_n\}$. Let $PB_n(S_g)$ (resp. $B_n(S_g)$) be the \emph{$n$-strand surface braid group}, i.e. the fundamental group of the space of ordered (resp. unordered) $n$ distinct points on $S_g$. By the \emph{generalized Birman exact sequence} (see e.g. \cite[Theorem 9.1]{BensonMargalit}), we have that $Ker(T\pi_{g,n})\cong PB_n(S_g)$ and $Ker(T\pi'_{g,n})\cong B_n(S_g)$. See \cite[Chapter 2.1]{lei1} for more details. We will prove the following proposition in the next subsection.
\begin{prop}For $g>1$ and $n\ge 0$. The following holds:

1) Every homomorphism $p$ satisfying the following diagram is either conjugate to a forgetful homomorphism $Tp_{g,n,i}$ by an element in ${\cal PI}_{g,n}$ or factors through $T\pi_{g,n}$, i.e. there exists $f$ such that $p=f\circ T\pi_{g,n}$.
\begin{equation}
\xymatrix{
1 \to PB_n(S_g)\ar[r]\ar[d]^R & {\cal PI}_{g,n} \ar[r]^{T\pi_{g,n}}\ar[d]^p& {\cal I}_g\ar[r]\ar[d]^=&  1 \\
1 \to\pi_1(S_g)\ar[r] &\text{\normalfont Mod}_{g,1}\ar[r]^{\pi_{g,1}}& \text{\normalfont Mod}_g \ar[r]& 1.}         
\label{diagram10}
\end{equation}

2) For $n>1$, every homomorphism $p'$ satisfying the following diagram factors through $T\pi_{g,n}'$, i.e. there exists $f'$ such that $p'=f'\circ T\pi_{g,n}'$
  \begin{equation}
\xymatrix{
1 \to B_n(S_g)\ar[r]\ar[d]^{R'} & {\cal I}_{g,n} \ar[r]^{T\pi_{g,n}'}\ar[d]^{p'}& {\cal I}_g\ar[r]\ar[d]^=&  1 \\
1 \to\pi_1(S_g)\ar[r] &\text{\normalfont Mod}_{g,1}\ar[r]^{\pi_{g,1}}&\text{\normalfont Mod}_g \ar[r]& 1.}    
\label{diagram20}
\end{equation}
\label{P}
\end{prop}

\begin{proof}[\bf Proof of Theorem \ref{Torelli} assuming Proposition \ref{P}]
By Theorem \ref{main}, the short exact sequence
\[
1\to \pi_1(S_g)\to {\cal I}_{g,1}\xrightarrow{\pi_{g,1}} {\cal I}_{g}\to 1\]
has no section. Therefore Proposition \ref{P} implies Theorem \ref{Torelli}.
\end{proof}

\subsection{The proof of Proposition \ref{P}}
The top exact sequence of diagram (\ref{diagram10}) gives us a representation $\rho_T:{\cal I}_g\to \text{Out}(PB_n(S_g))$. The following lemma describes a property of $\rho_T$. Let $p_i: PB_n(S_g)\to \pi_1(S_g)$ be the induced map on the fundamental groups of the forgetful map forgetting all points except the $i$th point.
\begin{lem}
Let $h>1$. For any surjective homomorphism $\phi: PB_n(S_g)\to F_h$, there exists an element $t\in {\cal I}_g$ such that $t(\text{Ker}(\phi))\neq \text{Ker}(\phi)$. 
\label{Johnson}
\end{lem}
\begin{proof}
By Theorem \cite[Theorem 1.5]{lei1}, any homomorphism $\phi:PB_n(S_g)\to F_h$ factors through some $p_i$. Thus we only need to deal with the case $n=1$. We will prove the lemma by contradiction. 

Suppose the opposite that there exists a surjective homomorphism $\phi:\pi_1(S_g)\to F_h$ \ such that for any element $e\in {\cal I}_g$, we have $e(\text{Ker}(\phi))= \text{Ker}(\phi)$. Since $\phi$ is surjective, the induced map on $H_1(\_\_,\mathbb{Z})$ is also surjective. Suppose that $a_1,a_2,...,a_h\in \pi_1(S_g)$ such that $\phi(a_1),...,\phi(a_h)$ generate $F_h$. Since the cup product $H^1(F_h,\mathbb{Z})\otimes H^1(F_h,\mathbb{Z})\xrightarrow{cup} H^2(F_h,\mathbb{Z})$ is trivial, the image of $\phi^*:H^1(F_h;\mathbb{Z})\to H^1(S_g;\mathbb{Z})$ is an isotropic subspace with dimension at most $g$. Thus we can find $b\in \pi_1(S_g)$ such that $\phi(b)=1$ and $[b]\neq 0\in H_1(S_g;\mathbb{Z})$. It is clear that $[b]$ and $\{a_1,...,a_h\}$ are linearly independent. Let $\pi^0=\pi_1(S_g)$, and $\pi^{n+1}=[\pi^n,\pi^0]$, we have the following exact sequence.
\[
1\to \pi^1/\pi^2\to \pi^0/\pi^2\to \pi^0/\pi^1\to 1\]

Let $H:=H_1(S_g;\mathbb{Z})$. Let $\omega=\sum_{j=1}^g a_j\wedge b_j$. We know that $\pi^1/\pi^2\cong \wedge^2 H/\mathbb{Z}\omega$, where the identification is given by $[x,y]\to x \wedge y$. Notice that ${\cal I}_g$ acts trivially on both $\pi^1/\pi^2$ and $\pi^0/\pi^1$ but nontrivially on $\pi^0/\pi^2$. The action is measured by the Johnson homomorphism $\tau:{\cal I}_g\to \text{Hom}(H,\wedge^2 H/\mathbb{Z}\omega)$; see \cite{MR579103} for more details. Let $t\in {\cal I}_g$. For $x\in H$, let $\tilde{x}\in \pi^0$ be a lift of $x$, i.e. $\tilde{x}$ maps $x$ under the map $\pi^0\to H$. The Johnson homomorphism is defined by $\tau(t)(x)=t(\tilde{x})\tilde{x}^{-1}\in \pi^1/\pi^2$. It is standard to check that $\tau(t)$ does not depend on the choice of lift $\tilde{x}$.

Johnson \cite[Theorem 1]{MR579103} proved that the image $\tau({\cal I}_g)=\wedge^3 H/H\subset \text{Hom}(H,\wedge^2 H/\mathbb{Z}\omega)$. Therefore there exists $t\in {\cal I}_g$ such that $\tau(t)(b)=a_1\wedge a_2$. By the definition of the Johnson homomorphism, we have that $t(b)b^{-1}=[a_1,a_2]T$, where $T\in \pi^2$. Since $\phi(b)=1$, we have that $\phi(t(b))=1$ by the assumption that $t(\text{Ker}(\phi))= \text{Ker}(\phi)$. As a result, $\phi([a_1,a_2])\phi(T)=1$. 

Let $F_h^1=[F_h,F_h]$ and $F_h^{n+1}=[F_h^n,F_h]$. We have that $\phi(\pi^n)\subset F_h^n$, which implies that $\phi(T)\in F_h^2$. However $\phi([a_1,a_2])\neq 1\in F_h^1/F_h^2$. This contradicts the fact that $\phi([a_1,a_2])=\phi(T)^{-1}$.

\end{proof}

We need the following lemma from \cite[Lemma 2.2]{handel}.
\begin{lem}
For $g>1$, a pseudo-Anosov element of {\normalfont Mod}$(S_{g,n})$ does not fix any nonperipheral isotopy class of curves including nonsimple curves.
\label{thur}
\end{lem}

Now we have all the ingredients to prove statement 1) in Proposition \ref{P}.
\begin{proof}[\bf Proof of 1) in Proposition \ref{P}]
For any $p:{\cal PI}_{g,n}\to {\cal I}_{g,1}$, we have that for $e\in {\cal PI}_{g,n}$ and $x\in PB_n(S_g)$,
\[
R(exe^{-1})=p(e)R(x)p(e)^{-1}.\]
Denote by $C_e$ the conjugation by $e$ in any group. This induces the following diagram:
\begin{equation}
\xymatrix{ PB_n(S_g)\ar[r]^{C_e}\ar[d]^R      &    PB_n(S_g)\ar[d]^R\\
\pi_1(S_g)\ar[r]^{C_{p(e)}}&   \pi_1(S_g) .  }
\label{CD}
\end{equation}

By \cite[Theorem 1.5]{lei1}, a homomorphism $R:PB_n(S_g)\to \pi_1(S_g)$ either factors through a forgetful homomorphism or has cyclic image. We break our discussion into the two cases.\\
\\
{\bf Case 1: Image$(R)\cong \mathbb{Z}$}\\
 In this case, the image is generated by $x\in \pi_1(S_g)$. By diagram \eqref{CD}, $C_{p(e)}$ preserves Image$(R)$ for any $e$. It is known that ${\cal I}_g$ contains pseudo-Anosov elements; see \cite[Corollary 14.3]{BensonMargalit}. By Lemma \ref{thur}, a pseudo-Anosov element does not preserve Image$(R)$. Therefore $R$ does not extend to $p$.\\
\\
{\bf Case 2: $R$ factors through a forgetful homomorphism $p_i$ and does not have cyclic image} \\
In this case, we have a homomorphism $S: \pi_1(S_g)\to \pi_1(S_g)$ such that $R=S \circ p_i$. If $S$ is a surjection, by the same reason as in the proof of \cite[Theorem 2.4]{lei1}, we know that $p$ is conjugate to $p_i$. If $S$ is not a surjection, then Image$(S)$ is a noncyclic free group. By Lemma \ref{Johnson} and diagram \ref{CD}, we know $R$ does not extend to $p$.
\end{proof}
To prove statement 2) in Proposition \ref{P}, we need the following lemma.
\begin{lem}
For $n>1$, the image of any homomorphism $B_n(S_g)\to \pi_1(S_g)$ is a free group.
\label{claim1}
\end{lem}
\begin{proof}
Suppose that there exists a homomorphism $\Phi :B_n(S_g)\to \pi_1(S_g)$ such that the image is not a free group. Then Image$(\Phi)\cong \pi_1(S_h)$ where $h\ge g$. After precomposing with the embedding $i:PB_n(S_g)\to B_n(S_g)$, we have a homomorphism $\Phi'$ $PB_n(S_g)\to \pi_1(S_g)$ with image a nontrivial finite index subgroup. By Theorem \cite[Theorem 5]{lei1}, the map $\Phi'$ factors through some $p_i$, but there is no surjection from $\pi_1(S_g)$ to a nontrivial finite index subgroup of $\pi_1(S_g)$. This is a contradiction. By the classification of subgroups of $\pi_1(S_g)$, the image of any homomorphism $B_n(S_g)\to \pi_1(S_g)$ is a free group.
\end{proof}

\begin{proof}[\bf Proof of 2) in Proposition \ref{P}]
By Claim \ref{claim1}, we know that $R'$ is not a surjection. Therefore, the image of $R'$ is either cyclic or a noncyclic free group. For the cyclic image case, we use the same argument as in the proof of \ref{P} to show that $R$ does not extend to $p$. In the case of noncyclic free group, by Lemma \ref{Johnson}, we know that $R$ does not extend to $p$ as well. 
\end{proof}

\subsection{A nonsplitting statement}
In this subsection, we will prove the following corollary using Theorem \ref{Torelli}.
\begin{cor}
For $g>1$ and $m>n$, the forgetful map $F_{g,m,n}:{\cal PI}_{g,m}\to {\cal PI}_{g,n}$ forgetting the last $m-n$ points does not have a section.
\label{cor}
\end{cor}
\begin{proof}
We only need to show that the $n$ sections of the bundle $Tu_{g,n}$ have nontrivial self-intersection. Then we cannot find $n+1$ disjoint sections on $Tu_{g,n}$. We restrict our attention to the subgroup $PB_n(S_g)$ of ${\cal PI}_{g,n}$. Let PConf$_n(S_g)$ be the space of n-tuples of distinct points on $S_g$. Since PConf$_n(S_g)=K(PB_n(S_g),1)$, we have that PConf$_n(S_g)$ is a subspace of ${\cal BPI}_{g,n}$. The bundle on PConf$_n(S_g)$ is the trivial bundle 
\[
\text{PConf}_n(S_g)\times S_g\xrightarrow{P_n} \text{PConf}_n(S_g)\] with $n$ sections $s_i(x_1,...,x_n)=(x_1,...,x_n,x_i)$. By Poincar\'e duality, the section is represented by a class in $H^2(\text{PConf}_n(S_g)\times S_g;\mathbb{Z})$. So the self-intersection of a section is a class in $H^4(\text{PConf}_n(S_g)\times S_g;\mathbb{Z})$. Let $p_i(x_1,...,x_n)=x_i$ be the projection of $\text{PConf}_n(S_g)$ to $S_g$. We have the following pullback diagram such that $s_i$ is the pullback of the diagonal section for the trivial bundle $P_1$.
\begin{equation}
\xymatrix{
\text{PConf}_n(S_g) \times S_g  \ar[r]^-{(p_i,id)}  \ar[d]^{P_n}   \pb  & S_g\times S_g\ar[d]^{P_1}\\
\text{PConf}_n(S_g)\ar[r]    & S_g  .}
\label{PB}
\end{equation}
Since $s_i$ is the pullback from the trivial bundle $P_1$, the self-intersection of $s_i$ is the pullback of the corresponding class $\gamma$ in $H^4(S_g\times S_g;\mathbb{Z})$. Let $[S_g]$ (resp. $[S_g\times S_g]$) be the fundamental class of $S_g$ (resp. $S_g\times S_g$). It is classical that the class is $\gamma=(2-2g)[S_g\times S_g]$. By the Gysin homomorphism, 
\[
p_{!}(p_i,id)^*(2-2g)[S_g\times S_g]=(2-2g)p_i^*[S_g]\in H^2(\text{PConf}_n(S_g);\mathbb{Z})
\] which is nonzero by the computation in \cite[Lemma 3.4]{lei1}.  

\end{proof}

\section{Another proof of Theorem \ref{main}}
We want to point out here that the punctured case can help us with the case of no punctures, i.e. Proposition \ref{P} can give us another proof of Theorem \ref{main}. Notice that the proof of Proposition \ref{P} does not depend on Theorem \ref{main}. Let ${\cal I}_{g,p}^b$ be the Torelli group of $S_{g,p}^b$, i.e. the subgroup of Mod$_{g,p}^b$ that acts trivially on $H^1(S_g;\mathbb{Z})$.

\begin{proof}[\bf Second proof of Theorem \ref{main}]
Again let $g>3$. We assume that the exact sequence (\ref{*}) has a splitting which is denoted by $\phi$ such that $F\circ \phi=id$. 

By Lemma \ref{haha}, the image $\phi(T_s)$ of $T_s$ the Dehn twist about a separating curve $s$ is $T_{s'}^nT_{s''}^{1-n}$ where $s'$ and $s''$ are curves on $S_{g,1}$ that are isotopic to $s$. Let $UTS_g$ be the unit tangent bundle of genus $g$ surface. Let $s$ be a separating curve that separates $S_g$ into two parts $C_1\cong S_p^1$ and $C_2\cong S_q^1$ such that $p,q \ge 2$. The combination of torelli groups of $C_1$ and $C_2$ gives us a subgroup $G$ of ${\cal I}_g$ satisfying the following short exact sequence.
\[
1\to \mathbb{Z}\xrightarrow{(T_s,T_s^{-1})} {\cal I}_p^1\times {\cal I}_q^1\to G\to 1
\]
The disk pushing subgroup is $\pi_1(UTS_p)\to {\cal I}_p^1$, i.e. see \cite[Page 118]{BensonMargalit}. The disk pushing subgroups of $C_1$ and $C_2$ give us a subgroup $A$ of $G$ satisfying the following short exact sequence.
\begin{equation}
1\to \mathbb{Z}\xrightarrow{(T_s,T_s^{-1})} \pi_1(UTS_p)\times \pi_1(UTS_q) \to A\to 1
\label{euler}
\end{equation}
\begin{claim}
$\phi(T_s)=T_{s'}$ for a curve $s'$ on $S_{g,1}$ that is isotopic to $s$.
\end{claim}
\begin{proof}
We have already proved this result in the proof of Theorem \ref{main}, Case 1. Here we give another proof using the Euler class. By Lemma \ref{haha}, we have that $\phi(T_s)=T_{s'}^nT_{s''}^{1-n}$. We only need to prove that CRS$(\phi(T_s))$ only contains one curve.

Suppose the opposite that CRS$(\phi(T_s))$ contains two curves $s'$ and $s''$ such that they are isotopic to $s$. Then $\phi(A)$ is in the centralizer of of $\phi(T_s)$. The centralizer $\phi(T_{s})$ is the subgroup of ${\cal I}_{g,1}$ that fixes $s'$ and $s''$. Since $\phi(A)$ also satisfies the fact that it maps to $A$ after forgetting $*$. We know that $\phi(A)\subset \pi_1(UTS_p)\times \pi_1(UTS_q)$. Since by computation, 
\[
\text{dim}H^2(UTS_p\times UTS_q;\mathbb{Q})=\text{dim}H^2(S_p\times S_q;\mathbb{Q})=\text{dim}H^2(A;\mathbb{Q}),
\]
we know that the Euler class of \eqref{euler} is nonzero. Therefore \eqref{euler} does not split which proves the claim.
\end{proof}

Since $T_s$ commutes with each element of $G$, we have that $T_{s'}$ commutes with each element of $\phi(G)$. Therefore, $\phi(G)$ is a subgroup of the centralizer of $T_{s'}$. The centralizer $C_{{\cal I}_{g,1}}(T_{s'})$ of $T_{s'}$ is the subgroup of ${\cal I}_{g,1}$ that fixes $s'$. Since the two components of $S_g-s'$ are not homeomorphic, any element in $C_{{\cal I}_{g,1}}(T_{s'})$ has to fix the two components. Therefore $C_{{\cal I}_{g,1}}(T_{s'})$ satisfies the following exact sequence
\[
1\to \mathbb{Z}\xrightarrow{} {\cal I}_p^1\times {\cal I}_{q,1}^1\to  C_{{\cal I}_{g,1}}(T_{s'})\to 1.
\]
Therefore we have a section of $F:{\cal I}_{q,1}^1\to {\cal I}_q^1$ which maps $T_{s}$ to $T_{s'}$. This section gives a section of $F_{q,2,1}$ in the following commutative diagram.
\begin{equation}
\xymatrix{
1\ar[r]  &  \mathbb{Z}\ar[r]\ar[d] &  {\cal I}_{q,1}^1\ar[r]\ar[d]^F  &     {\cal PI}_{q,2}   \ar[r] \ar[d]^{F_{q,2,1}} &1\\
1\ar[r]  &  \mathbb{Z}\ar[r]&  {\cal I}_{q}^1\ar[r]  &     {\cal I}_{q,1}   \ar[r]  &1.}
\end{equation}
However, we already prove that $F_{q,2,1}$ does not have a section in Corollary \ref{cor}, this implies that $F$ does not have a section. The statement follows.
\end{proof}

\small
\bibliography{citing}{}
\vspace{5mm}
\hfill \break
Dept. of Mathematics, University of Chicago

E-mail: chenlei@math.uchicago.edu

\end{document}